\documentclass[11pt]{article}
\usepackage{amssymb, amsmath, hyperref, array,color}
\usepackage{tikz}
\usetikzlibrary{arrows,arrows.meta,decorations.markings, positioning}

\newtheorem{defn}{Definition}
\newtheorem{notn}[defn]{Notation}
\newtheorem{lemma}[defn]{Lemma}
\newtheorem{proposition}[defn]{Proposition}
\newtheorem{theorem}[defn]{Theorem}
\newtheorem{corollary}[defn]{Corollary}
\newtheorem{remark}[defn]{Remark}
\newtheorem{example}[defn]{Example}

           {\vspace{3.3mm}
           \noindent{\bf #1}\it}%
           {\vspace{3.3mm}}

\newenvironment{proof}[1]{
  \trivlist \item[\hskip \labelsep{\it #1}]}{\hfill\mbox{$\square$}
  \endtrivlist}

\def\Z{\mathbb{Z}}
\def\Q{\mathbb{Q}}

\def\R{{\rm{\bf R}}}
\def\C{{\rm{\bf C}}}
\def\D{{\rm{\bf D}}}

\def\cX{{\cal X}}
\def\cY{{\cal Y}}
\def\cZ{{\cal Z}}

\def\sRes{{\rm sR}}
\def\sResP{{\rm sResP}}
\def\IVT{{\rm{[IVT]}}}
\def\NnS{{\rm{[NnS]}}}
\def\OD{{\rm{[OD]}}}
\def\FTA{{\rm{[FTA]}}}
\def\Ind{{\rm{Ind}}}
\def\sign{{\rm{sign}}}
\def\Rem{{\rm{Rem}}}

\def\Prem{{\rm{Prem}}}

\def\im{{\rm{im}\,}}

\title{Quantitative Fundamental Theorem of Algebra}

\author{Daniel Perrucci$^{\flat}$\thanks{{\scriptsize Partially supported by the Argentinian grants} {\footnotesize UBACYT 20020160100039BA} 
{\scriptsize and} {\footnotesize PIP 11220130100527CO CO\-NI\-CET}}  \quad  Marie-Fran\c{c}oise Roy$^{\sharp}$ \\[5mm]
{\small ${\flat}$ Departamento de Matem\'atica, FCEN, Universidad de Buenos Aires}\\
{\small and IMAS UBA-CONICET,}\\
{\small Ciudad Universitaria, 1428 Buenos Aires, Argentina}\\ 
{\small ${\sharp}$ IRMAR (UMR CNRS 6625), Universit\'e de Rennes 1,} \\
{\small Campus de Beaulieu, 35042 Rennes Cedex,  France}}

\begin{document}

\maketitle

\begin{abstract}
Using subresultants, we modify a 
real-algebraic
proof due to 
 Eisermann of the 
Fundamental Theorem of Algebra (\FTA) to obtain the following quantitative information: 
in order to prove the \FTA \ for polynomials of degree $d$, 
the Intermediate Value  Theorem (\IVT)
 is required to hold only for real polynomials of degree at most $d^2$. 
We also 
explain
that the classical proof due to Laplace requires \IVT \ for real polynomials of exponential degree. These quantitative results highlight the difference in nature of these two proofs.
\end{abstract}

\bigskip

\noindent  {\small \textbf{Keywords:}  Fundamental Theorem of Algebra, Intermediate Value Theorem, Cauchy Index, Winding Number, Subresultant Polynomials, Sturm Chains.} 

\medskip

\noindent  {\small \textbf{AMS subject classifications:} 14P99, 12D10, 12D15. }

\section{Introduction}

Let $(\R, \le)$ be an ordered field. 
The fact that $\R$ admits an order compatible with the field
structure implies that ${\rm char}(\R) = 0$ and 
therefore $\R$ has an infinite number of elements. 
It also implies that $-1$ is not a square in 
$\R$ and, consequently, $\R[T]/\langle T^2 + 1 \rangle$ $= \R[i] = \C$ 
is an algebraic field extension of $\R$ of degree 2.

We consider the following properties on $(\R, \le)$. 
\begin{itemize}
\item $\IVT$ (\emph{Intermediate Value Theorem}): for every polynomial $F \in \R[X]$ and
every $a, b$ in $\R$ with $a < b$ and $F(a)F(b) < 0$, there exists
$c \in \R$ with $a < c < b$ such that $F(c) = 0$. 

\item $\NnS$ (\emph{Non-negative elements are Squares}): for every $a \in \R$ with $a \ge 0$, there exists $c \in \R$
such that $a = c^2$. 

\item $\OD$ (\emph{An odd degree polynomial has a root}): for every polynomial $F \in \R[X]$ of odd degree, 
there exists $c \in \R$ such that $F(c) = 0$.

\item $\FTA$ (\emph{Fundamental Theorem of Algebra}): for every polynomial $F \in \C[Z] \setminus \C$, 
there exists $z \in \C$ such that $F(z) = 0$
(i.e., $\C$ is an algebraically closed field). 

\end{itemize}

If $\le$ denotes the usual order over the real numbers, $(\mathbb{R}, \le)$ and $(\mathbb{R}_{\rm alg}, \le)$ are typical examples 
of ordered fields 
satisfying all the properties above
(where $\mathbb{R}_{\rm alg}$ is the set of real algebraic numbers), whereas $(\mathbb{Q}, \le)$ is a 
typical example of an ordered field 
satisfying none of the properties above.

The next theorem is a classical result in real algebraic geometry (see for instance
\cite[Chapter 1]{BCR}).

\begin{theorem}\label{th:wellknown}
Let $(\R, \le)$ be an ordered field. The following conditions are equivalent:
\begin{enumerate}
\item[a)] $(\R, \le)$ satisfies \IVT.
\item[b)] $(\R, \le)$ satisfies \NnS \ and \OD.
\item[c)] $(\R, \le)$ satisfies \FTA.
\end{enumerate}
\end{theorem}

If $(\R, \le)$ satisfies these conditions, then it is easy to see that 
the field order $\le$ on $\R$ is unique and $\R$ is said to be a {\bf real closed field}. 

We sketch briefly a proof of Theorem \ref{th:wellknown}, which is essentially 
Laplace's proof \cite{Lap}.

\begin{proof}{Sketch of the proof of Theorem \ref{th:wellknown}:}
Proving that \IVT \ implies \NnS \ is very simple: for 
$a = 0$ we take $c = 0$;  and 
for $a > 0$ we consider the polynomial $F :=  X^2 - a \in \R[X]$ and
notice that $F(0) < 0$ and $F(a+1) > 0$, then \IVT  \
ensures the existence of a $c \in \R$ such that $F(c) = 0$, or 
equivalently, $a = c^2$. In fact, adding the condition $c \ge 0$, it
is  easy to prove the uniqueness of such $c$.

In a similar way, in order to prove that \IVT \ implies \OD \ 
we only need to note that an odd degree polynomial necessarily changes its sign
when evaluated at $a$ and $-a$ with $a \in \R$ big enough.

The proof that \NnS \ and \OD \ imply \FTA \ is much more sophisticated. 
To prove that $F \in \C[Z] \setminus \C$ 
has a root in $\C$, 
it is enough to prove that the polynomial $F \overline{F} \in \R[Z] \setminus \R$ 
has a 
root $w$ in $\C$ (where $\overline{F}$ means
the polynomial obtained from $F$ by usual conjugation in $\C$ of the 
coefficients of $F$); in this case either $w$ or $\overline{w}$ is a root of $F$.
Now, in order to show that an arbitrary polynomial $G \in \R[Z] \setminus \R$ of degree $d$
has a root in $\C$, the proof proceeds by induction on the highest 
value of $k$ such that $2^k$ divides
$d$. In the base case, which is $k = 0$ (and therefore odd $d$), the existence of a root of $G$
in $\R \subset \C$ is ensured by \OD. For $k \ge 1$ (and therefore even $d$), the existence 
of a root of $G$ in  $\C$ is ensured by a clever argument
involving \NnS \ and the fact that every polynomial 
in $\R[Z]$ of degree $\binom{d}{2}$ has a root in $\C$. 
Note that the highest power of $2$
dividing $\binom{d}{2} = \frac12d(d-1)$ is $2^{k-1}$ and then the inductive hypothesis holds. 

Finally, assuming \FTA, it 
is possible to prove that the irreducible elements in the unique factorization domain $\R[X]$ have 
degree $1$ or $2$ and that the irreducible monic elements in $\R[X]$ of 
degree $2$ are positive when evaluated at any $r \in \R$. From these facts, 
\IVT \ holds easily. 
\end{proof}

The main concern in the present work is the following question:
assuming that \IVT \ holds for $(\R, \le)$,
if we take a fixed value of $d\in \Z_{\ge 1}$ and
we only want to prove that 
every polynomial in $\C[Z] \setminus \C$ of degree
less than or equal to $d$ has a root in $\C$, which
is the highest degree of a polynomial in $\R[X]$ for
which \emph{we need} the Intermediate Value Theorem to hold?

With the aim of stating our problem precisely, we consider for each $d \in \Z_{\ge 1}$, the following properties on $(\R, \le)$.
\begin{itemize}
\item $\IVT_d$: for every polynomial $F \in \R[X]$ with $\deg F \le d$ and
every $a, b$ in $\R$ with $a < b$ and $F(a)F(b) < 0$, there exists
$c \in \R$ with $a < c < b$ such that $F(c) = 0$. 

\item $\FTA_d$: for every polynomial $F \in \C[Z] \setminus \C$ 
with $\deg F \le d$, 
there exists $z \in \C$ such that $F(z) = 0$. 
\end{itemize}

We can now restate our main concern as follows:
$$
\hbox{Given } d\in \Z_{\ge 1}, 
\hbox{ which is the lowest value of } \alpha(d) \in \Z_{\ge 1} 
\hbox{ for which }
\IVT_{\alpha(d)} \hbox{ implies } \FTA_d?
$$

In order to evaluate 
from this new quantitative point of view
the  proof of Theorem \ref{th:wellknown} we sketched,
we 
define the following functions:
\begin{notn}
Let $\beta, \gamma: \Z_{\ge 1} \to \Z_{\ge 1}$ defined as follows:
$$\begin{array}{rcl}
\beta (d) &:=&\left \{ \begin{array}{ll}
d & \hbox{if } d \hbox{ is odd}, \cr
\beta\big(\binom{d}2 \big) & \hbox{if } d \hbox{ is even,}
                                           \end{array} \right. \cr
\gamma(d) &:=& \max_{1 \le e \le d}\{\beta(2e)\}.
\end{array} 
                                           $$
\end{notn}

Note that $\gamma(1) = \beta(2) = 1$ and  for $d \ge 2$ we have that $\gamma(d) \ge \beta(4) = 15$.
Note also that $\gamma$ is a non-decreasing function, whereas the 
behavior of $\beta$ is rather chaotic.

First, we have that $\FTA_1$ holds even under no assumptions on $(\R, \le)$. Then,                                        
for a fixed $d \ge 2$ and a polynomial 
$F \in \C[Z]$ with $e = \deg F \le d$, 
in order to be able to apply the proof of Theorem \ref{th:wellknown}
we need to ensure 
\NnS \ and the fact that the Intermediate Value Theorem holds
for polynomials in $\R[X]$ of degree $\beta(2e)$. 
Since $\IVT_2$ implies \NnS, we have that
$$
\IVT_{\gamma(d)} \hbox{ implies } \FTA_d.
$$
The final conclusion is that $\alpha(d) \le \gamma(d)$. 

Now we want to exhibit explicit bounds for $\gamma$. 
It is possible to prove that for $d \in \Z, d \ge 4$, if $d = 2^ks$ with $k \in \Z_{\ge 0}$ and odd
$s \in \Z_{\ge 1}$
then  
$$
\frac83 \left(\frac34 2^{k-1}s  \right)^{2^k} \le \beta(d) \le 2 \left(2^{k-1}s  \right)^{2^k} \le 
2 \left(\frac{d}2 \right)^{d}.
$$
Then, for $d \in \Z, d \ge 4$, we have 
$$
\gamma(d)  \le   2 d^{2d}. 
$$
Also, by taking $k' := \lfloor \log_2 d \rfloor$, since $2^{k'} \le d < 2^{k'+1}$,
$$
\left(\frac38\right)^{d-1}d^d
=
\frac83 \left(\frac34 \frac{d}2  \right)^{d}
<
\frac83 \left(\frac34 2^{k'}  \right)^{2^{k'+1}}
\le
\beta(2^{k'+1})
\le \gamma(d).
$$

In this way, we know that $\gamma$ is bounded from below and above by 
exponential functions. This leads to an exponential upper bound
for $\alpha$, which 
cannot be avoided 
as long as we keep attached to the proof we sketched of Theorem \ref{th:wellknown}.

The exponential value of $\gamma(d)$ plays a significant
role in the bounds obtained in a recent joint work of the authors with 
Henri Lombardi, giving a new constructive proof for Hilbert 17-th problem and Positivstellensatz and providing
elementary recursive degree bounds \cite{LPR}.
Exploring other algebraic proofs of \FTA \ from a quantitative point of view might be a first step in the improvement of the results of \cite{LPR}. This hope is part of our motivation in this paper.

In \cite{Eis}, Michael Eisermann found a proof of the Fundamental Theorem of 
Algebra which is also
valid in any real closed field, 
but in opposition to Laplace's proof which is purely algebraic, 
has a large 
real-algebraic geometry flavor.
 This 
 proof of Eisermann can be seen as a real-algebraic adaptation of one of the classical proofs of the
 Fundamental Theorem of Algebra using winding numbers and homotopy
 (see \cite[Chapter 8]{FiRo}).
 One of the main ingredients of Eisermann's proof is  the
$\emph{Cauchy index}$ of two polynomials
which, roughly speaking, is the number of jumps from $-\infty$ to $+\infty$ minus
the number of jumps from $+\infty$ to $-\infty$ that the function 
associated to their quotient
has in a given interval. 
From the fact that the base ordered field $(\R, \le)$ satisfies $\IVT$, 
it follows an 
$\emph{inversion formula}$ which implies that Cauchy indices can be computed 
by counting sign variations on Sturm chains.
Another of the main ingredients of Eisermann's proof is the 
fact that the
$\emph{winding number}$ of a complex function on a rectangle, 
which counts the number of zeros of the function in the 
given rectangle, 
can be computed in a completely real-algebraic way by means of 
Cauchy indices on the boundary of the rectangle.
One of the most intricate steps in Eisermann's proof is to prove the Main Lemma 
(see \cite[Lemma 5.3]{Eis}): 
 if a polynomial does not vanish in a rectangle, then the associated winding number is zero.
This is achieved by a clever cancellation of terms for a suitable division of the rectangle
under consideration.
A crucial property for this cancellation is that, considering in the bivariate case 
one variable as the main variable and the second variable as a parameter, the 
(pseudo-)remainder sequence produces  Sturm chains
when specializing the parameter but also when specializing the main variable. 
Finally, by means of 
algebraic homotopy-like tools, the proof follows by computing 
the winding number in a well-known
special case.
Then the conclusion follows.

Our strategy is similar to the one by Eisermann, and also uses
the
$\emph{Cauchy index}$ of two polynomials and the $\emph{winding number}$ of a complex function on a rectangle.
The main new ingredient is that we use subresultants rather than (pesudo)-remainder sequences
to compute the Cauchy index. In order to be able to do this, we introduce the notion of $(\sigma, \tau)$-chain, which 
is a generalization of the notion of Sturm chain and prove that it can be used to compute Cauchy indices.
In case the base ordered field $(\R, \le)$ satisfies $\IVT_{d^2}$, 
it follows a refinement of the 
$\emph{inversion formula}$ which implies that Cauchy indices can be computed 
by counting sign variations on  $(\sigma, \tau)$-chains, when the degrees of the polynomials in the  $(\sigma, \tau)$-chains are all bounded by $d^2$.
Again,  the most intricate step in our proof of the Quantitatve Fundamental Theorme of Algebra is to prove the Quantitative Main Lemma  (Lemma \ref{prop:non_van_zero}):  if $\IVT_{d^2}$ holds and
if a polynomial of degree $d$ does not vanish in a rectangle, then the associated  winding number is zero.
This is also achieved by a cancellation of terms for a suitable   division of the rectangle
under consideration. 
As before, a crucial property for this cancellation is that, considering in the bivariate case one variable as the main variable and the second variable as a parameter, the subresultant sequence produces  $(\sigma, \tau)$-Sturm chains
when specializing the parameter but also when specializing the main variable. 
By using degree bounds on suresultant polynomials, we obtain that 
intermediate polynomials relevant to the proof are all of degree bounded by $d^2$.

Using this strategy, we prove the
following theorem which is our main result.

\begin{theorem}\label{th:main} For $d \in \Z_{\ge 1}$
$$
\IVT_{d^2} \hbox{ implies } \FTA_d.
$$
\end{theorem}
In other words, Theorem \ref{th:main} is equivalent to saying that 
$\alpha(d) \le d^2$.
Since $d \mapsto d^2$ is a polynomial function (actually $d^2 \le \gamma(d)$ 
for all  $d \in \Z_{\ge 1}$), 
our result
highlights the difference in nature between Laplace's proof and our modification of Eisermann's proof.

\section{Preliminaries}\label{sc:prelim}

In subsection \ref{subsec:cauch_index} and  subsection \ref{subsec:winding} we introduce Cauchy indices and winding numbers. In subsection \ref{subsec:boundedIVT} we extend results by Eisermann on Cauchy indices and winding numbers \cite{Eis} to the case where the Intermediate Value Theorem  holds only for polynomials of bounded degrees. Finally in subsection \ref{sec:subres} we give the needed preliminaries about subresultants.

\subsection{Cauchy index}\label{subsec:cauch_index} 
As said in the introduction, the Cauchy index of 
two polynomials $Q$ and $P$
on an interval 
is, roughly speaking, the number of jumps from $-\infty$ to $+\infty$ minus
the number of jumps from $+\infty$ to $-\infty$ that the function 
associated to their quotient $\displaystyle{\frac{Q}P}$
has in this interval.
We recall now the precise definition of Cauchy index following \cite[Section 3]{Eis}.

\begin{notn}
Let $x \in \R$, we denote the sign of $x$ by 
$$  
\sign(x) := \left\{ \begin{array}{ll}
1  & \hbox{if } x > 0, \\
0  & \hbox{if } x = 0, \\
-1  & \hbox{if } x < 0. \\
\end{array}\right.
$$
\end{notn}

\begin{defn}\label{defn:CI_at_a_point} Let $x \in \R$ and $P, Q \in \R[X]$.
\begin{itemize}
\item 
If $P, Q \in \R[X]\setminus\{0\}$, the polynomials $P$ and $Q$ can be written uniquely as

$$
P = (X - x)^{\mu(x)}  \widetilde P,
$$
$$
Q = (X - x)^{\nu(x)} \widetilde Q,
$$
with $\mu(x),\nu(x)\in \Z_{\ge 0}$ and  $\widetilde P(x) \ne 0, \widetilde Q(x) \ne 0$. 

 For $\varepsilon \in \{+1, -1\}$, define

$$
\Ind_x^{\varepsilon}(Q,P) := \left\{
\begin{array}{ll}
 \frac{1}2  \sign(\widetilde{Q}(x) \widetilde{P}(x)) & \hbox{if } \varepsilon = +1 \hbox{ and } \mu(x)>\nu(x),\\[3mm]
 \frac{1}2  (-1)^{\mu(x)-\nu(x)}  \sign(\widetilde{Q}(x) \widetilde{P}(x)) & \hbox{if } 
 \varepsilon = -1 \hbox{ and } \mu(x)>\nu(x),\\[3mm]
0 & \hbox{otherwise}.
 \end{array}
\right.
$$

\item If $P = 0$ or $Q=0$, define 
$$\Ind_x^{\varepsilon}(Q,P) :=0.$$

\item The Cauchy index of $(Q,P)$ at $x$ is 
$$
\Ind_x(Q, P) := 
\Ind_x^{+} (Q, P)
-
\Ind_x^{-} (Q, P).
$$
\end{itemize}
\end{defn}

We illustrate this notion considering the graph of the function
$\displaystyle{\frac{Q}P}$ around $x$ in each different case.

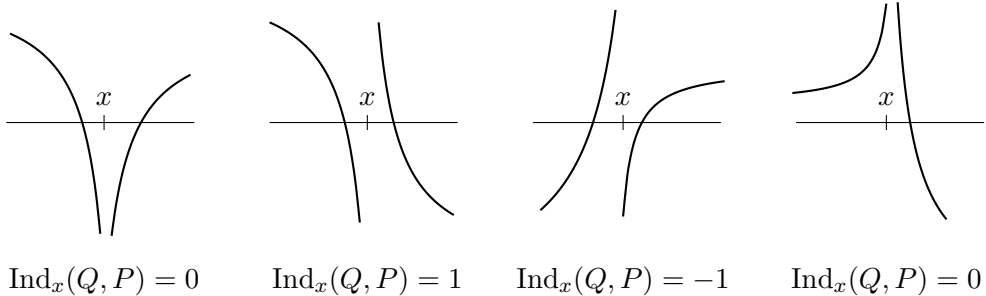
\begin{center}
\begin{tikzpicture}
      \draw[-] (-6.5,0) -- (-4,0);
      \draw[-] (-3,0) -- (-0.5,0);
      \draw[-] (0.5,0) -- (3,0);
      \draw[-] (4,0) -- (6.5,0);
      \draw[-] (-5.2,-0.1) -- (-5.2,0.1) node[above] {$x$} ;
      \draw[-] (-1.7,-0.1) -- (-1.7,0.1) node[above] {$x$} ;
      \draw[-] (1.7,-0.1) -- (1.7,0.1) node[above] {$x$} ;
      \draw[-] (5.2,-0.1) -- (5.2,0.1) node[above] {$x$} ;
      \draw[line width=0.8pt, domain=-6.5:-5.3,smooth,variable=\x] plot ({\x+0.05},{1/(\x+5)+1.85});
      \draw[line width=0.8pt, domain=-5.05:-4,smooth,variable=\x,] plot ({\x-0.05},{-1/(\x+5.4)+1.35});
      \draw[line width=0.8pt, domain=-3:-1.8,smooth,variable=\x] plot ({\x},{1/(\x+1.5)+2});
      \draw[line width=0.8pt, domain=-1.5:-0.5,smooth,variable=\x] plot ({\x-0.05},{1/(\x+1.8)-2});
      \draw[line width=0.8pt, domain=0.5:1.5,smooth,variable=\x] plot ({\x+0.1},{-2/(\x-2)-2.5});
      \draw[line width=0.8pt, domain=1.65:3,smooth,variable=\x] plot ({\x+0.05},{-0.3/(\x-1.5)+0.75});
      \draw[line width=0.8pt, domain=4:5.25,smooth,variable=\x,] plot ({\x-0.05},{-0.2/(\x-5.4)+0.25});
      \draw[line width=0.8pt, domain=5.25:5.9,smooth,variable=\x] plot ({\x+0.1},{1/(\x-5)-2.4}); 
      \node at (-5.2,-2.1) {$\Ind_x(Q,P) = 0 $};
      \node at (-1.7,-2.1) {$\Ind_x(Q,P) = 1 $};
      \node at (1.7,-2.1) {$\Ind_x(Q,P) = -1 $};
      \node at (5.2,-2.1) {$\Ind_x(Q,P) = 0 $};
      
      \end{tikzpicture}
\end{center}

\begin{defn} Let $a, b \in \R$ and $P, Q \in \R[X]$.
\begin{itemize}
\item If $a < b$ and $P, Q \ne 0$, the Cauchy index of 
$(Q,P)$  on the interval $[a, b]$ is 
$$
\Ind_a^b(Q,P) := 
\Ind_a^+(Q,P)  + \sum_{x \in (a,b) }  \Ind_x (Q,P)
- \Ind_b^-(Q,P),
$$
where the sum is well-defined since only roots $x$ of $P$ in $(a, b)$ contribute. 
\item  If $a > b$ and $P, Q \ne 0$, 
$$ \Ind_a^b(Q,P) := - \Ind_b^a(Q,P).$$
\item In every other case, 
$$ \Ind_a^b(Q,P) := 0.$$
\end{itemize}
\end{defn}

In the following picture we consider again the graph of the function
$\displaystyle{\frac{Q}P}$, this time in $[a, b]$.

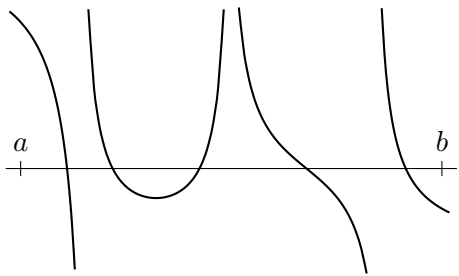
\begin{center}
\begin{tikzpicture}
      \draw[-] (-7,0) -- (-1,0);
      \draw[-] (1,0) -- (7,0);
      \draw[-] (-6.8,-0.1) -- (-6.8,0.1) node[above] {$a$};
      \draw[-] (-1.2,-0.1) -- (-1.2,0.1) node[above] {$b$};
      \draw[-] (1.2,-0.1) -- (1.2,0.1) node[above] {$a$};
      \draw[-] (6.8,-0.1) -- (6.8,0.1) node[above] {$b$} ;
      \draw[line width=0.8pt, domain=-7:-6.13,smooth,variable=\x] plot ({\x+0.05},{1/(\x+5.9)+3});
      \draw[line width=0.8pt, domain=-5.9:-4.1,smooth,variable=\x,] plot ({\x},{-1/((\x+6.05)*(\x+3.95))-1.3});
      \draw[line width=0.8pt, domain=-3.9:-2.2,smooth,variable=\x] plot ({\x},{(\x+3)/((\x+4.11)*(\x+1.9))});
      \draw[line width=0.8pt, domain=-1.9:-1,smooth,variable=\x] plot ({\x-0.1},{(\x+3)/((\x+4.1)*(\x+2.05))-1.2});
      \draw[line width=0.8pt, domain=1:2.46,smooth,variable=\x] plot ({\x},{-(\x-3)/((\x-4.1)*(\x-2.55))-1.5});
      \draw[line width=0.8pt, domain=2.5:5,smooth,variable=\x] plot ({\x},{-0.65*(\x-3.5)/((\x-5.2)*(\x-2.35))+0.3});
      \draw[line width=0.8pt, domain=5.2:6.69,smooth,variable=\x] plot ({\x},{-0.5*(\x-6)/((\x-5.05)*(\x-6.8))+0.3});
      \node at (-4,-2.1) {$\Ind_a^b(Q,P) = 1 + 0 +1 = 2 $};
      \node at (4,-2.1) {$\Ind_a^b(Q,P) = -1 -1 -\frac12 = -\frac52$};
      \end{tikzpicture}
\end{center}

Note that 
 the Cauchy index of 
a pair of polynomials
on an interval belongs to $\frac12 \Z$ and it is not necessarily an integer number.

\begin{remark} \label{rem:deg_ind} 
\rm{If both $P$ and $Q$ are multiplied by $S\in \R[X]\setminus \{0\}$, it is clear that 
$\Ind_a^b(Q,P)=\Ind_a^b(QS, PS)$, so when $P\not=0$ the Cauchy index 
is associated to the rational function $\displaystyle{\frac{Q}{P}}$ rather than to the pair of polynomials $(Q,P)$. 
However, when $P=0$, it is convenient for us to define also the Cauchy index, even if the rational function $\displaystyle{\frac{Q}{P}}$ does not make sense.
This is the reason why we use the notation $\Ind_a^b(Q,P)$ in all cases.}
\end{remark}

\begin{remark}
\rm{Even though it is not reflected in the notation, the field $\R$ plays a fundamental role 
in the definition of the Cauchy index. For instance, 
consider $P := X^2 - 2, Q := 1 \in \Q[X] \subset \mathbb{R}[X]$.
If we take $\R = \Q$  we have
$$
\Ind_{1}^2(Q,P) = 0, 
$$
whereas if we take $\R = \mathbb{R}$ we have
$$
\Ind_{1}^2(Q,P) = 1. 
$$}
\end{remark}

\begin{remark}\label{rem:lin_resc}
\rm{Cauchy index is invariant by affine change of variables:  
given any affine function $\ell:[a, b] \to \R$, and $P, Q \in \R[X]$, 
$$
\Ind_{\ell(a)}^{\ell(b)}(Q,P) = \Ind_a^b(Q \circ \ell,P \circ \ell). 
$$
(By an affine function, we mean a function of type $\ell(X) = cX + d$
with $c, d \in \R$.)
}
\end{remark}

\begin{remark}\label{rem:add_int}
\rm{Cauchy index is additive  on intervals: 
given any $a, c_1, \dots, c_k, b \in \R$
and $P, Q \in \R[X]$, 
$$
\Ind_a^b(Q,P) = 
\Ind_a^{c_1}(Q,P) + \sum_{1 \le i \le k-1} 
\Ind_{c_i}^{c_{i+1}}(Q,P)
+ \Ind_{c_k}^{b}(Q,P).
$$}
\end{remark}

\subsection{Winding number}\label{subsec:winding}

From now on, we consider the usual identification $\C \sim \R^2$. 

The \emph{winding number} of a closed curve in $\C$
is a classical object which counts, 
by means of an 
analytic expression, 
the number of 
counterclockwise turns of the curve around the origin.  
In this paper, we will always restrict to curves which are 
the image of a polynomial function $F \in \C[X, Y]$ on the border of a rectangle $\Gamma \subset \R^2$
whose sides are parallel to the axis. 
For curves of this type, we recall the algebraic definition of 
\emph{winding number} following \cite{Eis}.
Note that the border of $\Gamma$, denoted by $\partial \Gamma$, is simply the union of four segments.

\begin{notn} 
For $F \in \C[X, Y]$, we denote
$F_{\rm re}$ and  $F_{\rm im}$ the real and imaginary parts of $F$, i.e.
the unique
polynomials in $\R[X, Y]$ such that the identity 
$$
F(X, Y) =
F_{\rm re}(X,Y)+iF_{\rm im}(X,Y)
$$
in $\C[X, Y]$ holds.
\end{notn}

\begin{defn} \label{def:winding_number}
Let $x_0, x_1, y_0, y_1 \in \R$ with $x_0 < x_1$ and $y_0 < y_1$
and let $\Gamma \subset \R^2$ be the rectangle 
$\Gamma := [x_0, x_1] \times [y_0, y_1]$.
For $F \in \C[X,Y]$ the \emph{winding number} of $F$ on $\partial \Gamma$
is defined as
$$\begin{array}{rcl}
w(F \, | \, \partial \Gamma) &:= & \frac12 \left(
\Ind_{x_0}^{x_1}(F_{\rm re}(X, y_0),F_{\rm im}(X, y_0))
+
\Ind_{y_0}^{y_1}
(F_{\rm re}(x_1,Y),F_{\rm im}(x_1,Y))
\right.\\[3mm]
&&+
\left.
\Ind_{x_1}^{x_0}
(F_{\rm  re}(X, y_1),F_{\rm im}(X, y_1))
+
\Ind_{y_1}^{y_0}
(F_{\rm re}(x_0,Y),F_{\rm im}(x_0,Y))
\right).
\end{array}
$$
\end{defn}

Notice that 
it follows from
the definition of winding number that 
we are going through $\partial \Gamma$ following the counterclock sense. 
The idea behind this algebraic definition is 
to count one half of a turn each time this curve crosses the $X$-axis from quadrant IV to I or from 
quadrant II to III, and minus one half of a turn each time it crosses the $X$-axis from quadrant I 
to IV or from 
quadrant III to II. 
Since these crossings coincide with jumps of the rational function $\displaystyle{\frac{F_{\rm re}}{F_{\rm im}}}$ from $-\infty$ to $+\infty$
and from $+\infty$ to $-\infty$ respectively, the Cauchy index is an appropriate algebraic tool to
count the number of turns counterclockwise, which is (when  $F$ does not vanish on $\partial \Gamma$)
the classical definition of the winding number.

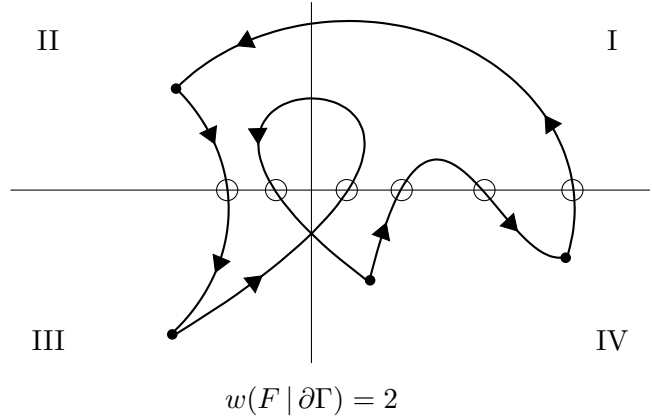
\begin{center}
\begin{tikzpicture}
      \draw[-] (-4,0) -- (4.5,0);
      \draw[-] (0,-2.3) -- (0,2.5);
      \node at (4,2) {I};
      \node at (-3.5,2) {II};
      \node at (-3.5,-2) {III};
      \node at (4,-2) {IV}; 
      \draw[line width=0.8pt, domain=-0.5:1.04, smooth,variable=\x] plot ({3*sin(100*\x) +0.5},{2.5*cos(100*\x)-0.25});	
      \draw[{Triangle[scale=1.8]}-, domain=0.6:0.62,smooth,variable=\x] plot ({3*sin(100*\x) +0.5},{2.5*cos(100*\x)-0.25});
      \draw[{Triangle[scale=1.8]}-, domain=-0.3:-0.28,smooth,variable=\x] plot ({3*sin(100*\x) +0.5},{2.5*cos(100*\x)-0.25});
      \draw[line width=0.8pt, domain=-1.32:-0.5,smooth,variable=\x] plot ({-4.6-3*sin(100*\x)+0.5},{2.5*cos(100*\x)-0.25});
      \draw[{Triangle[scale=1.8]}-, domain=-0.7:-0.62,smooth,variable=\x] plot ({-4.6-3*sin(100*\x)+0.5},{2.5*cos(100*\x)-0.25});
      \draw[{Triangle[scale=1.8]}-, domain=-1.1:-1.05,smooth,variable=\x] plot ({-4.6-3*sin(100*\x)+0.5},{2.5*cos(100*\x)-0.25});
      \draw[line width=0.8pt, domain=-1.63:1.43,smooth,variable=\x] plot ({\x^3 - 1.5*\x},{1-1.2*(\x*\x)+0.22});
      \draw[-{Triangle[scale=1.8]}, domain=-1.43:-1.4,smooth,variable=\x] plot ({\x^3 - 1.5*\x},{1-1.2*(\x*\x)+0.22});
      \draw[-{Triangle[scale=1.8]}, domain=0.7:0.72,smooth,variable=\x] plot ({\x^3 - 1.5*\x},{1-1.2*(\x*\x)+0.22});
      \draw[line width=0.8pt, domain=0.77:3.4,smooth,variable=\x] plot ({\x},{0.6*(\x-0.85)*(\x-3.3)*(\x-3.3)-0.9});
      \draw[-{Triangle[scale=1.8]}, domain=0.97:1,smooth,variable=\x] plot ({\x},{0.6*(\x-0.85)*(\x-3.3)*(\x-3.3)-0.9});
      \draw[-{Triangle[scale=1.8]}, domain=2.7:2.75,smooth,variable=\x] plot ({\x},{0.6*(\x-0.85)*(\x-3.3)*(\x-3.3)-0.9});
      \fill (3.38, -0.9) circle[radius=2pt];
      \fill (-1.8, 1.35) circle[radius=2pt];
      \fill (-1.85, -1.92) circle[radius=2pt];
      \fill (0.78, -1.2) circle[radius=2pt];
      \draw (3.47, 0) circle[radius=4pt];
      \draw (2.3, 0) circle[radius=4pt];
      \draw (1.2, 0) circle[radius=4pt];
      \draw (0.47, 0) circle[radius=4pt];
      \draw (-0.47, 0) circle[radius=4pt];
      \draw (-1.12, 0) circle[radius=4pt];
      \node at (0,-2.8) {$w(F \, | \, \partial \Gamma) = 2$}; 
\end{tikzpicture}
\end{center}

Along the paper we will follow the convention of 
using $X, Y$ and $T$ 
for 
real variables, i.e. variables
that will only be eventually evaluated 
at elements of $\R$, 
and $Z$ 
for a 
complex variable, i.e. a variable
that will be eventually evaluated 
at arbitrary elements of $\C$.

To $F \in \C[Z]$
we associate $\bar F(X,Y) := F(X + iY) \in \C[X,Y]$. 
Abusing slightly notation, we denote 
$F_{\rm re}, F_{\rm im},w(F \, | \, \partial \Gamma)$ for $\bar F_{\rm re}, \bar F_{\rm im},w(\bar F \, | \, \partial \Gamma)$.

\begin{example}\label{exm:wind_linear} (See \cite[Proposition 4.4]{Eis}) Let 
$\Gamma := [x_0, x_1] \times [y_0, y_1] \subset \R^2$.
For $z \in \C$, we have
$$
w(Z-z \, | \, \partial \Gamma) =
\left\{
\begin{array}{cl}
1 & \hbox{if } z \hbox{ is in the interior of }  \Gamma, \cr        
1/2 & \hbox{if } z \hbox{ is in one of the edges of }  \Gamma, \cr        
1/4 & \hbox{if } z \hbox{ is a vertex of }  \Gamma, \cr        
0 & \hbox{if } z \hbox{ is in the exterior of } \Gamma. \cr        
\end{array}
\right.
$$
\end{example}

\begin{lemma}\label{lem:grid_partition}
Let $\Gamma := [x_0, x_1] \times [y_0, y_1] \subset \R^2$ and consider a grid 
partition 
of $\Gamma$ into a finite number of rectangles $\Gamma_1, \dots, \Gamma_s$. 
For $F \in \C[X, Y]$, we have
$$
w(F \, | \, \partial \Gamma) = \sum_{1 \le i \le s}w(F \, | \, \partial \Gamma_i).
$$
\end{lemma}

\begin{proof}{Proof:}
After replacing the winding number of $F$ on $\partial\Gamma_1, \dots, \partial\Gamma_s$
by its 
definition, 
along each edge in the interior of $\Gamma$   we have to 
add and subtract the Cauchy index of the 
same 
pair of polynomials, which adds up to zero. On the other hand, using the additivity 
of Cauchy index on intervals (Remark \ref{rem:add_int}), adding on the 
remaining edges we obtain the winding number of $F$ on $\partial\Gamma$.
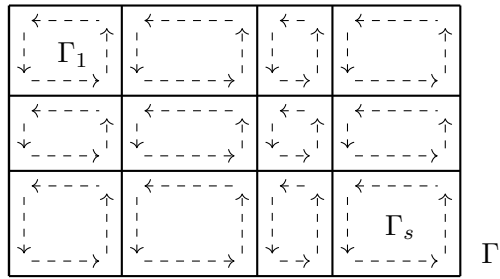
\begin{center}
\begin{tikzpicture}
      \draw[line width=0.9pt,-] (-3,-1.8) -- (-3,1.8);
      \draw[line width=0.8pt,-] (-1.5,-1.8) -- (-1.5,1.8);
      \draw[line width=0.8pt,-] (0.3,-1.8) -- (0.3,1.8);
      \draw[line width=0.8pt,-] (1.3,-1.8) -- (1.3,1.8);
      \draw[line width=0.8pt,-] (3,-1.8) -- (3,1.8);
      \draw[<-, dashed] (-2.8,-1.5) -- (-2.8,-0.7);
      \draw[<-, dashed] (-2.8,-0.1) -- (-2.8,0.3);	
      \draw[<-, dashed] (-2.8,0.9) -- (-2.8,1.5);
      \draw[->, dashed] (-1.7,-1.5) -- (-1.7,-0.7);
      \draw[->, dashed] (-1.7,-0.1) -- (-1.7,0.3);	
      \draw[->, dashed] (-1.7,0.9) -- (-1.7,1.5);
      \draw[<-, dashed] (-1.3,-1.5) -- (-1.3,-0.7);
      \draw[<-, dashed] (-1.3,-0.1) -- (-1.3,0.3);	
      \draw[<-, dashed] (-1.3,0.9) -- (-1.3,1.5);
      \draw[->, dashed] (0.1,-1.5) -- (0.1,-0.7);
      \draw[->, dashed] (0.1,-0.1) -- (0.1,0.3);	
      \draw[->, dashed] (0.1,0.9) -- (0.1,1.5);
      \draw[<-, dashed] (0.5,-1.5) -- (0.5,-0.7);
      \draw[<-, dashed] (0.5,-0.1) -- (0.5,0.3);	
      \draw[<-, dashed] (0.5,0.9) -- (0.5,1.5);
      \draw[->, dashed] (1.1,-1.5) -- (1.1,-0.7);
      \draw[->, dashed] (1.1,-0.1) -- (1.1,0.3);	
      \draw[->, dashed] (1.1,0.9) -- (1.1,1.5);
      \draw[<-, dashed] (1.5,-1.5) -- (1.5,-0.7);
      \draw[<-, dashed] (1.5,-0.1) -- (1.5,0.3);	
      \draw[<-, dashed] (1.5,0.9) -- (1.5,1.5);
      \draw[->, dashed] (2.8,-1.5) -- (2.8,-0.7);
      \draw[->, dashed] (2.8,-0.1) -- (2.8,0.3);	
      \draw[->, dashed] (2.8,0.9) -- (2.8,1.5);
      \draw[line width=0.8pt,-] (-3,-1.8) -- (3,-1.8);
      \draw[line width=0.8pt,-] (-3, -0.4) -- (3,-0.4);
      \draw[line width=0.8pt,-] (-3, 0.6) -- (3,0.6);
      \draw[line width=0.8pt,-] (-3, 1.8) -- (3,1.8);
      \draw[<-, dashed] (-2.7,1.6) -- (-1.8,1.6);	
      \draw[<-, dashed] (-1.2,1.6) -- (0,1.6);
      \draw[<-, dashed] (0.6,1.6) -- (1,1.6);	
      \draw[<-, dashed] (1.6,1.6) -- (2.7,1.6);
      \draw[->, dashed] (-2.7,0.8) -- (-1.8,0.8);	
      \draw[->, dashed] (-1.2,0.8) -- (0,0.8);
      \draw[->, dashed] (0.6,0.8) -- (1,0.8);	
      \draw[->, dashed] (1.6,0.8) -- (2.7,0.8);
      \draw[<-, dashed] (-2.7,0.4) -- (-1.8,0.4);	
      \draw[<-, dashed] (-1.2,0.4) -- (0,0.4);
      \draw[<-, dashed] (0.6,0.4) -- (1,0.4);	
      \draw[<-, dashed] (1.6,0.4) -- (2.7,0.4);
      \draw[->, dashed] (-2.7,-0.2) -- (-1.8,-0.2);	
      \draw[->, dashed] (-1.2,-0.2) -- (0,-0.2);
      \draw[->, dashed] (0.6,-0.2) -- (1,-0.2);	
      \draw[->, dashed] (1.6,-0.2) -- (2.7,-0.2);
      \draw[<-, dashed] (-2.7,-0.6) -- (-1.8,-0.6);	
      \draw[<-, dashed] (-1.2,-0.6) -- (0,-0.6);
      \draw[<-, dashed] (0.6,-0.6) -- (1,-0.6);	
      \draw[<-, dashed] (1.6,-0.6) -- (2.7,-0.6);
      \draw[->, dashed] (-2.7,-1.6) -- (-1.8,-1.6);	
      \draw[->, dashed] (-1.2,-1.6) -- (0,-1.6);
      \draw[->, dashed] (0.6,-1.6) -- (1,-1.6);	
      \draw[->, dashed] (1.6,-1.6) -- (2.7,-1.6);
      \node at (-2.15,1.15) {$\Gamma_1$};
      \node at (2.2,-1.15) {$\Gamma_s$};
      \node at (3.4,-1.5) {$\Gamma$};
\end{tikzpicture}
\end{center}
\end{proof}

To finish this subsection, we prove the following lemma which will play an important role at the end of Section \ref{sc:main} when
applying homotopy tools.

\begin{lemma}\label{lem:homot_cube}
Let $x_0, x_1, y_0, y_1, t_0, t_1 \in \R$ with $x_0 < x_1, y_0 < y_1$ and $t_0 < t_1$.
Let $\Gamma_T := [x_0, x_1] \times [y_0, y_1], 
\Gamma_Y := [x_0, x_1] \times [t_0, t_1], \Gamma_X := [y_0, y_1] \times [t_0, t_1]  
\subset \R^2$.
For $F \in \C[X, Y, T]$, we have
$$
\begin{array}{cccccccr}
-&w(F(X, Y, t_0) \, | \, \partial \Gamma_T) 
&+&
w(F(X, y_0, T) \, | \, \partial \Gamma_Y)
&-&
w(F(x_0, Y, T) \, | \, \partial \Gamma_X) & \\[2mm]
+&  
w(F(X, Y, t_1) \, | \, \partial \Gamma_T)  
&-& 
w(F(X, y_1, T) \, | \, \partial \Gamma_Y)  
&+&  
w(F(x_1, Y, T) \, | \, \partial \Gamma_X) & = & 0. 
\end{array}
$$
Therefore, if 
$$
w(F(X, y_0, T) \, | \, \partial \Gamma_Y)
= 
w(F(x_0, Y, T) \, | \, \partial \Gamma_X) 
=  
w(F(X, y_1, T) \, | \, \partial \Gamma_Y)  
=
w(F(x_1, Y, T) \, | \, \partial \Gamma_X) 
= 0$$
then
$$
w(F(X, Y, t_0) \, | \, \partial \Gamma_T)  
=
w(F(X, Y, t_1) \, | \, \partial \Gamma_T).
$$
\end{lemma}
\begin{proof}{Proof:}
Consider the rectangular parallelepiped $[x_0, x_1] \times [y_0, y_1] \times [t_0, t_1] \subset \R^3$. 
After replacing each winding number by its definition, 
along each edge of this parallelepiped we have to add and subtract the Cauchy index of the 
same pair of polynomials; therefore obtaining $0$ as the final result. 

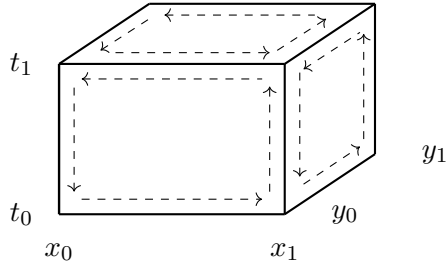
\begin{center}
\begin{tikzpicture}
      \draw[line width=0.9pt,-] (-1.5,-1) -- (-1.5,1);
      \draw[line width=0.8pt,-] (1.5,-1) -- (1.5,1);
      \draw[line width=0.8pt,-] (2.7, -0.2) -- (2.7,1.8);
      \draw[line width=0.8pt,-] (-1.5,-1) -- (1.5,-1);
      \draw[line width=0.8pt,-] (-1.5, 1) -- (1.5,1);
      \draw[line width=0.8pt,-] (-0.3,1.8) -- (2.7,1.8);
      \draw[line width=0.8pt,-] (1.5,-1) -- (2.7,-0.2);
      \draw[line width=0.8pt,-] (1.5,1) -- (2.7,1.8);
      \draw[line width=0.8pt,-] (-1.5,1) -- (-0.3,1.8);
      \draw[<-, dashed] (-1.3,-0.7) -- (-1.3,0.7);
      \draw[->, dashed] (1.3,-0.7) -- (1.3,0.7);
      \draw[<-, dashed] (1.7,-0.5) -- (1.7,0.9);
      \draw[->, dashed] (2.55, 0) -- (2.55,1.4);
      \draw[->, dashed] (-1.2,-0.8) -- (1.2,-0.8);	
      \draw[<-, dashed] (-1.2,0.8) -- (1.2,0.8);
      \draw[->, dashed] (-0.75,1.15) -- (1.3,1.15);	
      \draw[<-, dashed] (-0.1,1.65) -- (1.9,1.65);
      \draw[->, dashed] (1.75,-0.60) -- (2.5,-0.1);
      \draw[<-, dashed] (1.75,0.93) -- (2.5,1.42);
      \draw[->, dashed] (1.4,1.18) -- (2.05,1.6);
      \draw[<-, dashed] (-0.95,1.18) -- (-0.3,1.6);
      \node at (-1.5,-1.5) {$x_0$};
      \node at (1.5,-1.5) {$x_1$};
      \node at (2.3,-1) {$y_0$};
      \node at (3.5,-0.2) {$y_1$};
      \node at (-2,-1) {$t_0$};
      \node at (-2,1) {$t_1$};      
\end{tikzpicture}
\end{center}
\end{proof}

\subsection{The intermediate value property for polynomials of bounded degree}\label{subsec:boundedIVT}

Our main goal in this paper is to prove that $\IVT_{d^2}$ implies $\FTA_d$.
So, from now, 
we take a fixed value of $d \in \Z_{\ge 1}$ and we suppose that $(\R, \le)$ 
is an ordered field satisfying $\IVT_{d^2}$ but not necessarily $\IVT$.
Since $\FTA_1$ holds even under no assumptions on $(\R, \le)$, 
we suppose $d\ge 2$.

Note that the current assumption on $(\R, \le)$ is rather 
subtle,
since 
for instance, it is only for $P \in \R[X]$ with $\deg P \le d^2$ that we 
can claim that if $P$ has no roots
on an interval $I \subset \R$, then $P$ has constant sign (different from $0$) on $I$.

The purpose of this section is to 
reexamine
some results from \cite{Eis} 
concerning the Cauchy index and the winding number
as well as to prove
that they still
hold in the present setting, despite the fact that our 
hypotheses are
weaker than 
in \cite{Eis}. 
More explicitly,  in \cite{Eis} the assumption is that 
$\R$ is a real closed field, and therefore it satisfies $\IVT$, 
whereas we only suppose $\IVT_{d^2}$.
Nevertheless, 
in the results reviewed in this section, following the proofs in \cite{Eis} or a
slight variation of it, it turns out that the Intermediate Value Theorem is applied
to polynomials of degree less than or equal to $d^2$, and this is enough to ensure that
these results
are still valid. For completeness, we include anyway full proofs of the statements in this section. 
We will use many times the following easy remark. 

\begin{remark}\label{rem:sign_easy}
\rm{For $x, y \in \{-1, 0, 1\}$ with $(x, y) \ne (0, 0)$}, 
$$
\sign(xy) = 1 - |x - y|. 
$$
\end{remark}

We introduce the following useful notation.

\begin{notn}\label{notn:sign_var}
Let $x \in \R$ and $P, Q \in \R[X]$,
we denote the sign variation of $(P, Q)$ at $x$ by
$$
{\rm Var}_x(P,Q):= \frac12  \Big|\sign(P(x)) - \sign(Q(x)) \Big|.
$$
For $a, b\in \R$,
we denote
by ${\rm Var}_a^b(P, Q)$
the sign variation of $(P,Q)$ at $a$ minus
the sign variation of $(P, Q)$ at $b$;  namely, 
$$
{\rm Var}_a^b(P, Q) :=
{\rm Var}_a(P, Q) -{\rm Var}_b(P, Q).
$$
\end{notn}

We first prove the following property, which is a refinement of the well known \emph{the inversion formula}.

\begin{proposition}\label{prop:inv_formula}
Let $a, b \in \R$ and $P, Q \in \R[X]$ with
$\deg P, \deg Q \le d^2$ and
such that $P$ and $Q$ 
have no common root in $[a, b]$.
Then 
$$
\Ind_a^b(Q,P) + 
\Ind_a^b(P,Q)
= 
{\rm Var}_a^b(P, Q)
$$
\end{proposition}

\begin{proof}{Proof:} We follow the arguments from \cite[Theorem 3.9]{Eis}. If $P = 0$ or $Q = 0$, 
since $P$ and $Q$ have no common root in $[a, b]$ we have that 
${\rm Var}_a^b(P, Q) = 0$ and the result holds. Suppose now that $P \ne 0$ and $Q \ne 0$. 
From the invariance by 
affine change of variables and the additivity on intervals of Cauchy index (Remarks \ref{rem:lin_resc}
and \ref{rem:add_int}) we suppose that $a$ is the only possible root of 
$P$ or $Q$ on $[a,b]$ and that $a$ is indeed a root of $P$. 

Write $$
P = (X - a)^{\mu(a)}  \widetilde P,
$$
with $\mu(a) \in \Z_{> 0}$ and  $\widetilde P(a) \ne 0$. 
Then using Remark \ref{rem:sign_easy} we have
$$
\begin{array}{rcl}
 \Ind_a^b(Q,P) + \Ind_a^b(P,Q) &
= &\frac12 \sign(Q(a)\widetilde P(a))  + 0\\[3mm]
&= &\frac12 - \frac12 \Big|\sign(\widetilde P(a)) - \sign(Q(a)) \Big|\\[3mm]
&= &\frac12 - \frac12 \Big|\sign(P(b)) - \sign(Q(b)) \Big|\\[3mm]
&= &\frac12 - {\rm Var}_b(P, Q)\\[3mm]
&= &{\rm Var}_a(P, Q) - {\rm Var}_b(P, Q)\\[3mm]
&= & {\rm Var}_a^b(P, Q).
\end{array}
$$
\end{proof}

Next proposition shows the additivity of the winding number with respect to the product of complex polynomials.

\begin{proposition}\label{prop:mult_wind} Let 
$\Gamma := [x_0, x_1] \times [y_0, y_1] \subset \R^2$ and
$F, G \in \C[X,Y]$ with $\deg FG \le d^2$ 
and such that $F$ and $G$ do not vanish in $\partial \Gamma$. Then
$$
w(FG \, | \, \partial \Gamma) = w(F \, | \, \partial \Gamma) + w(G \, | \, \partial \Gamma).
$$
\end{proposition}

The proof of Proposition \ref{prop:mult_wind} uses the next lemma as an auxiliary result.

\begin{lemma}\label{lem:aux_ind_prod}
Let $a, b \in \R$ and $P, Q, R, S \in \R[X]$ with 
$\deg (PR - QS),$ $\deg(PS + QR) \le d^2$ and such that 
$P$ and $Q$ have no common root in $[a, b]$ and 
$R$ and $S$ have no common root in $[a, b]$. Then
$$
\Ind_a^b\left(PR-QS, PS+QR\right) =$$
$$
\Ind_a^b\left(P,Q\right) + 
\Ind_a^b\left(R,S\right)
+ \frac12 \sign\Big(   \big((PS+QR)QS\big)(a)    \Big) 
- 
\frac12 \sign
   \Big(\big((PS+QR)QS\big)(b)  \Big).
$$
\end{lemma}

\begin{remark} \rm{In \cite[Theorem 4.5]{Eis} there is a statement with a slightly 
different formula and no assumption on polynomials $P, Q, R, S \in \R[X]$.
We observed that this formula does not hold for the case 
$a := 0, b := 1, P := 1, Q := X, R := X-1, S := X$. 
Notice that in this example, $P, Q, R, S$ actually meet our extra assumptions, 
but since $PS + QR = QS = X^2$, if we  deal with 
the rational 
function $PS + QR/QS$ 
as in \cite[Theorem 4.5]{Eis}, there is a simplification
which is the cause of the trouble. 
For this reason, we work with polynomials and not rational functions; but then 
the extra assumptions of not having common roots are necessary,
since common factors would not modify the Cauchy indices but could modify the signs
involved in the formula in Lemma \ref{lem:aux_ind_prod}.
To illustrate this situation, 
$a := 0, b := 1, P := X-1, Q := X(X-1), R := X-1, S := X$ would be a counterexample
if we made no assumptions.}
\end{remark}

\begin{proof}{Proof of Lemma \ref{lem:aux_ind_prod}:}

First, we prove that  the condition 
$\deg (PR - QS),$ $\deg(PS + QR) \le d^2$ implies
$\deg P, \deg Q,
\deg R, \deg S \le d^2$. If at least one of the polynomials $P$, $Q$, $R$ or $S$ equals $0$, then the degree bound 
on the other three polynomials clearly holds. If none of these polynomials equals $0$, 
suppose that the claim does not hold and let us look for a contradiction. Respectively denote by $p, q, r, s \in \R$ their leading coefficients. 
The fact that there is a degree drop in $PR - QS$ and $PS + QR$ with respect to one of the polynomials 
$P, Q, R$ or $S$ implies that $pr - qs = 0$ and $ps + qr =0$. Since 
$p, q, r, s \ne 0$, we deduce that $p^2 + q^2 = r^2 + s^2 = 0$, and this is not possible since 
$\R$ is an ordered field.

Now we continue to prove the lemma following the ideas in \cite[Theorem 4.5]{Eis}. 
If
$Q = 0$, $S = 0$ or $PS + QR = 0$, the result is immediate. 
If $P = 0$ or $R = 0$, the result
follows from Proposition \ref{prop:inv_formula} using Remark \ref{rem:sign_easy}. 
In every other case, from the invariance by 
affine change of variables and the additivity on intervals of Cauchy index (Remarks \ref{rem:lin_resc}
and \ref{rem:add_int}) we suppose that $a$ is the only possible root of 
$P, Q, R, S$ or $PS + QR$
on $[a,b]$. We consider several cases as follows.

\begin{itemize}

\item If $Q(a) \ne 0, S(a) \ne 0$ and $(PS + QR)(a) \ne 0$, then 
$$
\Ind_a^b\left(PR-QS, PS+QR\right) = 
\Ind_a^b\left(P,Q\right) = 
\Ind_a^b\left(R,S\right) = 0
$$
and
$$
\sign\Big(   \big((PS+QR)QS\big)(a)    \Big) 
 =  
\sign\Big(\big((PS+QR)QS\big)(b)  \Big)
$$
so the identity holds.

\item If $Q(a) \ne 0, S(a) \ne 0$ and $(PS + QR)(a) = 0$, then
$$
\Ind_a^b\left(P,Q\right) = 
\Ind_a^b\left(R,S\right) = 
\sign\Big(   \big((PS+QR)QS\big)(a)    \Big)  =  
0.
$$
On the other hand
$$
\frac{P(a)}{Q(a)} =  - \frac{R(a)}{S(a)},  
$$
so 
$$
(PR - QS)(a) = 
Q(a)S(a)\underbrace{\left(\frac{P(a)}{Q(a)}\frac{R(a)}{S(a)} - 1  \right)}_{<0} \ne 0
$$
Write $PS + QR = (X- a)^{\mu}T$ with $\mu \in \Z_{>0}$
and $T(a) \ne 0$. Note that $\sign(T(a)) = \sign(T(b)) = \sign((PS+QR)(b))$. 
Then we have
$$
\Ind_a^b\left(PR-QS, PS+QR\right) = -\frac12 \sign \big(Q(a)S(a)T(a)\big) = -\frac12 
\sign\Big(\big((PS+QR)QS\big)(b)  \Big)
$$
so the identity holds.

\item If $Q(a) = 0$ and $S(a) \ne 0$, since 
$P$ and $Q$ have no common root in $[a, b]$ then $(PS + QR)(a) \ne 0$
and we have that 
$$
\Ind_a^b\left(PR-QS, PS+QR\right) = 
\Ind_a^b\left(R,S\right) = 
\sign\Big(   \big((PS+QR)QS\big)(a)    \Big)  =  
0.
$$
Write $Q = (X- a)^{\mu}\widetilde Q$ with $\mu \in \Z_{>0}$
and $\widetilde Q(a) \ne 0$.
Then
$$
\Ind_a^b\left(P,Q\right) 
= \frac12 \sign\big(P(a)\widetilde Q(a)\big) 
= \frac12 \sign\Big(\big((PS+QR)\widetilde QS\big)(a)  \Big)
= \frac12 \sign\Big(\big((PS+QR)QS\big)(b)  \Big)
$$
so the identity holds.

\item If $Q(a) \ne 0$ and $S(a) = 0$ we proceed in a similar way to the previous case. 

\item If $Q(a) = 0$ and $S(a) = 0$, then $(PS + QR)(a) = 0$, and since
$P$ and $Q$ have no common root in $[a, b]$ and 
$R$ and $S$ have no common root in $[a, b]$, 
$P(a) \ne 0$, $R(a) \ne 0$.

Write $PS + QR = (X- a)^{\mu_0}T$ with $\mu_0 \in \Z_{>0}$
and $T(a) \ne 0$, $Q = (X- a)^{\mu_1}\widetilde Q$ with $\mu_1 \in \Z_{>0}$
and $\widetilde Q(a) \ne 0$ and
$S = (X- a)^{\mu_2}\widetilde S$ with $\mu_2 \in \Z_{>0}$
and $\widetilde S(a) \ne 0$.
We denote
$$
\begin{array}{rcl}
\sigma_1 & := & \sign( P(a) ) \in \{-1, 1\}, \cr
\sigma_2 & := & \sign( R(a) ) \in \{-1, 1\}, \cr
\sigma_3 & := & \sign(T(a) ) \in \{-1, 1\}, \cr
\sigma_4 & := & \sign( \widetilde Q(a) )\in \{-1, 1\}, \cr
\sigma_5 & := & \sign( \widetilde S(a) )\in \{-1, 1\}. \cr
\end{array}
$$
We need to prove that
$$
\sigma_1\sigma_2\sigma_3 = \sigma_1\sigma_4 + \sigma_2\sigma_5 - \sigma_3\sigma_4\sigma_5
$$
or, equivalently, 
\begin{equation}\label{eq:aux_prod_formula}
\big(\sigma_1\sigma_2 + \sigma_4\sigma_5\big)\sigma_3 = \sigma_1\sigma_4 + \sigma_2\sigma_5 
\end{equation}
We take into account that 
$\sigma_1 = \sign(P(b))$, 
$\sigma_2 = \sign(R(b))$, 
$\sigma_3 = \sign((PS+QR)(b))$, 
$\sigma_4 = \sign(Q(b))$ and 
$\sigma_5 = \sign(S(b))$ and we divide in cases as follows. 
\begin{itemize}
\item If $\sigma_1 = \sigma_5$ and $\sigma_2 = \sigma_4$, then $\sigma_3 = 1$ and 
equation (\ref{eq:aux_prod_formula})
holds. 
\item If $\sigma_1 = -\sigma_5$ and $\sigma_2 = -\sigma_4$ then $\sigma_3 = -1$ and
equation (\ref{eq:aux_prod_formula})
holds. 
\item In every other case, exactly three elements in the set $\{\sigma_1, \sigma_2, \sigma_4, \sigma_5\}$ are 
equal and the remaining one is different. Then 
$$
\sigma_1\sigma_2 + \sigma_4\sigma_5 = \sigma_1\sigma_4 + \sigma_2\sigma_5 = 0
$$
and
equation (\ref{eq:aux_prod_formula})
holds. 
\end{itemize}

\end{itemize}

\end{proof}

\begin{proof}{Proof of Proposition \ref{prop:mult_wind}:}
The proof is done in as in 
\cite[Corollary 4.6 and Corollary 4.7]{Eis}. 
After replacing each winding number by its definition, we apply 
Lemma \ref{lem:aux_ind_prod} once on each side of $\partial \Gamma$.
For instance, on the bottom side, we take $a := x_0, b := x_1$, 
$P := F_{\rm re}(X, y_0), Q := F_{\rm im}(X, y_0), 
R:= G_{\rm re}(X, y_0)$ and
 $S := G_{\rm im}(X, y_0)$.
The identity in the lemma is obtained after checking that on
each vertex of $\partial \Gamma$, 
signs cancel after being added on one side and subtracted on the other side. 
\end{proof}

From Example \ref{exm:wind_linear} and Proposition \ref{prop:mult_wind} 
the following result is easily deduced. 

\begin{example}\label{exm:wind_monomial} For $e \in \Z_{\ge 1}$ with $e \le d^2$,
$$
w(Z^e \, | \, \partial \Gamma) = e.
$$
if $\Gamma \subset \R^2$ is a rectangle containing $0$ in its interior. 
\end{example}

Finally, we recall the property saying that the winding number 
vanishes in a \emph{small} rectangle 
around a non-zero of a polynomial.

\begin{proposition} \label{prop:non_van_zero_local} 
Let $(x, y) \in \R^2$ and $F \in \C[X,Y]$ with $\deg F \le
d^2$ and 
such that $F(x,y) \ne 0$. Then there exists $\delta \in \R, \delta > 0$ such that 
for every rectangle $\Gamma \subset [x - \delta, x + \delta] \times [y - \delta, y + \delta]
\subset \R^2$, $F$ does not vanish in $\Gamma$ and $w(F \, | \, \partial \Gamma) = 0$.
 
\end{proposition}

\begin{proof}{Proof:} 
First, since $\IVT_{d^2}$ holds, it is easy to see
that for every $a \in \R$ with $a \ge 0$ and every $n \in \Z_{\ge 1}, n \le d^2$, 
there is a unique $c \in \R$ such that $c \ge 0$ and $c^n = a$, 
which we note as $c = a^{1/n}$. It is clear that if $a >0$ then $a^{1/n} > 0$. 

Then we just follow the arguments from \cite[Lemma 5.2]{Eis}.  
We take
$$
G := \frac{i}{F(x, y)} F \in \C[X, Y]
$$
and we need to prove that 
there exists $\delta \in \R, \delta > 0$ such that 
$G_{\im} \in \R[X, Y]$ does not vanish in
$[x - \delta, x + \delta] \times [y - \delta, y + \delta] \subset \R^2$.
In this case, 
for every rectangle $\Gamma \subset [x - \delta, x + \delta] \times [y - \delta, y + \delta]$, 
we have that $G \in \C[X, Y]$ does not vanish in $\Gamma$ and  $w(G \, | \, \partial \Gamma) = 0$. 
Then the lemma follows using Proposition \ref{prop:mult_wind}.

Suppose now
$$
G_{\im} = \sum_{j = (j_1, j_2)\atop{j_1 + j_2 \le d^2}}c_j(X-x)^{j_1}(Y-y)^{j_2}.
$$
Since $G(x, y) = i$ we know that $G_{\im}(x, y) = 1$. If $G_{\im}$ is constant, then any positive value of $\delta$ works. 
Otherwise, taking
$$
\Delta = \frac12(d^2+1)(d^2+2)
$$
and
$$
\delta := \min \left\{ \left(\frac{1}{\Delta|c_j|}\right)^{\frac1{j_1+j_2}} \ | \  
j = (j_1, j_2), 1 \le j_1 + j_2 \le d^2, c_j \ne 0
\right\} > 0,
$$
for every $(z, w) \in [- \delta, \delta] \times [-\delta, \delta]$ we have
$$
G_{\im}(x+z , y+w) 
=  1 + \displaystyle{\sum_{j = (j_1, j_2)\atop{1 \le j_1 + j_2 \le d^2}}c_jz^{j_1}w^{j_2}} 
\ge  1 - \displaystyle{\sum_{j = (j_1, j_2)\atop{1 \le j_1 + j_2 \le d^2}}|c_j|\delta^{j_1+j_2}}
\ge  1 - \displaystyle{\sum_{j = (j_1, j_2)\atop{1 \le j_1 + j_2 \le d^2}}\frac1{\Delta}} 
=  \frac1\Delta 
>  0.
$$
\end{proof}

\subsection{Subresultant polynomials}\label{sec:subres}

Let $\D$ be an integral domain.
The subresultant polynomial sequence of two polynomials $P, Q \in \D[X]$ 
is a sequence of polynomials in $\D[X]$ which
contains the
classical Sylvester resultant of $P, Q$;
more specifically, the last subresultant polynomial, which actually belongs to $\D$, 
coincides up to sign with 
the Sylvester resultant. 
Even though the subresultant polynomials of $P$ and $Q$ are 
defined
in a completely 
different way, they are closely related to the polynomials appearing in the remainder 
sequence of $P$ and $Q$, as reflected in the Structure Theorem of Subresultants
(Theorem \ref{thm:structure_thm_subresultants}). It can be 
proved that the 
behavior 
of their coefficients is better controlled than the 
behavior of the 
coefficients of the polynomials in the remainder sequence, and for this reason, they constitute a 
widely used tool in gcd computation, real root counting and many other problems in computational algebra. 
In Section 3, we will use subresultants in the particular case of $\D = \R[Y]$ 
and the good 
behavior
of their coefficients 
implies  a good control of 
the degree in $Y$ (Proposition \ref{prop:degree_bound_res}), 
which will be a key point to obtain our 
main result.

We include now some definitions and properties concerning subresultants. We refer
the reader to \cite{BPRbook} for proofs and details.

\begin{defn}\label{def:subres} Let $P
, Q \in \D[X] \setminus \{0\}$ with $p := \deg P
 \ge 1$ 
and
$q := \deg Q
< p$. 

\begin{itemize}
\item For $0 \le
  j \le q$, the  Sylvester-Habicht matrix 
  ${\rm SyHa}_j (P, Q) \in \D^{(p + q - 2j) \times 
(p + q - j)}$ is the matrix 
  whose rows are the polynomials
  \[ X^{q - j - 1}P, \ldots, P, Q, \ldots, X^{p - j - 1}Q, \]
  expressed in the monomial basis $X^{p+q-j-1},\ldots,X,1$.
\item For $0 \le j \le q$, the $j$-th subresultant polynomial of 
$P$ and $Q$, ${\sResP}_j (P, Q) \in \D[X]$ 
  is the polynomial determinant of ${\rm SyHa}_j (P, Q)$, i.e.
  $$
  {\sResP}_j (P, Q):=\sum_{0 \le i \le j} \det({\rm SyHa}_{j,i}(P, Q)) \cdot X^i
  \in \D[X]
  $$
  where ${\rm SyHa}_{j,i}(P, Q)\in \D^{(p + q - 2j) \times 
(p + q - 2j)}$ is the matrix obtained by taking the
  $p+q-2j-1$ first
  columns 
  and the $(p+q-j-i)$-th column of  ${\rm SyHa}_j (P, Q)$.
 By convention, we
  extend this definition 
  with
  \begin{eqnarray*}
    {\sResP}_p (P, Q) & := & P \ \in \D[X],\\
    {\sResP}_{p - 1} (P, Q) & := & Q \ \in \D[X],\\
    {\sResP}_j (P, Q) & := & \ 0  \ \in \D[X] \qquad \hbox{ for } q < j < p - 1.
  \end{eqnarray*}

\item For $0 \le j \le q$, 
the $j$-th signed subresultant coefficient of $P$ and $Q$, ${\sRes}_j (P,
Q)\in \D$  is the coefficient of $X^j$ in
${\sResP}_j (P, Q)$. 
By convention, we extend this definition with  
\begin{eqnarray*}
    {\sRes}_p(P, Q) & := & 1 \ \in \D \qquad \hbox{(even if } P \hbox{ is not monic)},\\
    {\sRes}_j(P, Q) & := & 0 \ \in \D \qquad \hbox{ for } q < j \le p - 1.
  \end{eqnarray*}

\item For $0 \le j \le p$, ${\sResP}_j (P, Q)$ is said to be \emph{defective} 
if $\deg {\sResP}_j (P, Q) < j$ or, equivalently, if ${\sRes}_j (P, Q)
= 0$.

\end{itemize}

\end{defn}

We will also use the following notation.

\begin{notn}\label{notn:lead_non_def}
Let $P, Q
\in \D[X] \setminus \{0\}$ 
with $p := \deg P
\ge 1$ and $q := \deg Q
   < p$.
Let 
$(d_0,\ldots,d_s)$ be
the sequence of degrees of the non-defective subresultant polynomials of $P$ and 
$Q$ in decreasing order
(note that $d_0 = p$ and $d_1 = q$).
For $1 \le i \le s$, 
$$
T_{d_{i-1}-1}(P,Q):={\rm lcoeff}({\sResP}_{d_{i-1}-1}(P,Q)) \in \D\setminus \{0\}.
$$
We extend this notation with $T_{p}(P,Q):= 1 \in \D\setminus \{0\}$.
\end{notn}

The following theorem is one of the most important results in the theory 
of subresultants. This result has a long history \cite{Ha}. We quote its more recent form in \cite{BPRbook}, which is a slight improvement of
\cite{LiR}.

\begin{theorem}[Structure Theorem of Subresultants]
\label{thm:structure_thm_subresultants} 
Let $P, Q
\in \D[X] \setminus \{0\}$ 
with $p := \deg P
  \ge 1$ and $q := \deg Q
   < p$.
Let 
$(d_0,\ldots,d_s)$ be
the sequence of degrees of the non-defective subresultant polynomials of $P$ and 
$Q$ in decreasing order
and let 
$d_{-1} := p+1$. 
Then
\begin{itemize}
\item for $1 \le i \le s$, 
$$
{\sResP}_{d_{i-1}-2}(P,Q) = \dots = 
{\sResP}_{d_{i} + 1}(P,Q) = 0 \in \D[X]
$$
and 
${\sResP}_{d_{i-1}-1}(P,Q)$ and 
${\sResP}_{d_{i}}(P,Q)$ are proportional. More precisely,
$$
{\sRes}_{d_{i}}(P,Q) \cdot {\sResP}_{d_{i-1}-1}(P,Q) =  
{T}_{d_{i-1}-1}(P,Q) \cdot {\sResP}_{d_{i}}(P,Q) \in \D[X]
$$
with
$$
{\sRes}_{d_{i}}(P,Q)
=
(-1)^{\frac12(d_{i-1}-d_i)(d_{i-1}-d_i -1)} 
\frac{{T}_{d_{i-1}-1}(P,Q)^{d_{i-1}-d_i}}
{{\sRes}_{d_{i-1}}(P,Q)^{d_{i-1}-d_i-1}} \in \D.
$$
This implies $\deg {\sResP}_{d_{i-1}-1}(P,Q) = d_i$.

\item for $1 \le i \le s$, 
$$\begin{array}{cl}
&T_{d_{i-2}-1}(P,Q) \cdot 
{\sRes}_{d_{i-1}}(P,Q)
\cdot {\sResP}_{d_{i}-1}(P, Q)\\[2mm]
=&
-
{\rm Rem}
\left(
T_{d_{i-1}-1}(P,Q) \cdot 
{\sRes}_{d_{i}}(P,Q)
\cdot {\sResP}_{d_{i-2}-1}(P, Q)
,
{\sResP}_{d_{i-1}-1}(P, Q) 
\right) 
\in \D[X]
\end{array}
$$
and
$${\rm Quot}
\left(
T_{d_{i-1}-1}(P,Q) \cdot 
{\sRes}_{d_{i}}(P,Q)
\cdot {\sResP}_{d_{i-2}-1}(P, Q)
,
{\sResP}_{d_{i-1}-1}(P, Q) 
\right) 
\in \D[X]
$$
(where ${\rm Rem}$ and ${\rm Quot}$ means the remainder and quotient
in the euclidean division in ${\rm qf}(\D)[X]$ of the first polynomial by the second polynomial).

\item ${\sResP}_{d_{s-1}-1}(P, Q) \in \D[X]$ 
and ${\sResP}_{d_{s}}(P, Q) \in \D[X]$
are the 
greatest common divisor of $P$ and $Q$ in ${\rm qf}(\D)[X]$ 
multiplied by elements in $\D$. 
In addition, if $d_s > 0$ then
$$
{\sResP}_{d_{s}-1}(P, Q) = \dots = {\sResP}_{0}(P, Q) = 0 \in \D[X].
$$
\end{itemize}
\end{theorem}

$$\begin{array}{rl}
\rule{8cm}{0.4pt} & \sResP_{d_0} = \sResP_p = P \\
\rule{6.2cm}{0.4pt} & \sResP_{d_0 - 1} = \sResP_{p-1} = Q \\
0 & \\[-1mm]
\vdots  {\hspace{0.05cm} }& \\
0 & \\[-2mm]
\rule{6.2cm}{0.4pt} & \sResP_{d_1} = \sResP_{q} \\
\rule{4.2cm}{0.4pt}& \sResP_{d_1 - 1}  \\
0 & \\ [-1mm]
\vdots {\hspace{0.05cm} } & \\[-2mm]
\vdots {\hspace{0.05cm} } & \\
0 & \\ [-2mm]
\rule{4.2cm}{0.4pt} & \sResP_{d_2} \\
\vdots {\hspace{0.05cm} } & \\[-2mm]
\vdots {\hspace{0.05cm} } & \\[-1mm]
\rule{1.3cm}{0.4pt}& \sResP_{d_{s-1} - 1}  \\
0 & \\ [-1mm]
\vdots {\hspace{0.05cm} } & \\[-2mm]
\vdots {\hspace{0.05cm} } & \\
0 & \\ [-2mm]
\rule{1.3cm}{0.4pt} & \sResP_{d_s}  \\
0 & \\ [-1mm]
\vdots {\hspace{0.05cm} } & \\
0 & \\ [-2mm]
\end{array}
$$

\begin{proof}{Proof:} See
\cite[Chapter 8]{BPRbook}.
\end{proof}

As said before, in Section \ref{sc:main} we will use subresultants in the particular 
case of $\D = \R[Y]$ and we will need some degree bounds which we develop here.

\begin{proposition}\label{prop:degree_bound_res}
Let $P, Q \in \R[X, Y] \setminus \{0\}$ with $p := \deg_X P
 \ge 1$, $q := \deg_X  Q
  < p$ and total degree $\deg P, \deg Q \le d$ (with $d \ge p$). 
We consider subresultants with respect to variable $X$ 
(this is to say, 
following Definition \ref{def:subres} 
we take $\D = \R[Y]$).
For $0 \le j \le q$ and 
$0 \le i \le j$, the degree in $Y$ of the coefficient of $X^i$ 
in ${\sResP}_j(P, Q) \in \R[Y][X]$ is bounded by $d^2$.
\end{proposition}

\begin{proof}{Proof:} Let $P = \sum_{0 \le i \le p}a_i(Y)X^i$ and 
$Q = \sum_{0 \le i \le q}b_i(Y)X^i$; then $\deg_Y a_i(Y), \deg_Y b_i(Y) \le d - i$.
By definition, 
$$
  {\sResP}_j (P, Q)=\sum_{0 \le i \le j} \det({\rm SyHa}_{j,i}(P, Q)) \cdot X^i
  \in \R[Y][X], 
$$
  where ${\rm SyHa}_{j,i}(P, Q) \in \R[Y]^{(p + q - 2j) \times 
(p + q - 2j)}$ is the matrix obtained by taking the
  $p+q-2j-1$ first
  columns 
  and the $(p+q-j-i)$-th column of  
  ${\rm SyHa}_j (P, Q)\in \R[Y]^{(p + q - 2j) \times 
(p + q - j)}$.
By defining $a_i(Y) = 0$ if $i \ge p+1$ or $i \le -1$ and $b_i(Y) = 0$ if $i \ge q+1$
or $i \le -1$ we have 
that for $1 \le k \le p + q - 2j$ and $1 \le \ell \le p + q - j$,
$$
\left({\rm SyHa}_{j}(P, Q)\right)_{k \ell} = 
\left\{
\begin{array}{ll}
a_{p + k - \ell}(Y) &  \hbox{if } k \le q-j,  \cr
b_{p +2q - 2j +1 -k - \ell }(Y) &  \hbox{if } k \ge q-j+1. \cr
\end{array}
\right.
$$
The proof can be completed by 
bounding the degree of
any possible nonzero product 
of entries of ${\rm SyHa}_{j,i}(P, Q)$ with
one element per row and column.
We obtain that the degree in $Y$ of the coefficient of $X^i$ 
in ${\sResP}_j(P, Q) \in \R[Y][X]$ is bounded by
$$d(p+q-2j) -pq + j^2 +j - i \le d(p+q) -pq \le d^2.$$
\end{proof}

\section{Counting complex roots}\label{sc:main}

In this section we introduce $(\sigma,\tau)$-chains, 
develop suitable 
generalizations
of results from \cite{Eis}
and 
prove Theorem \ref{th:main}.
As said before, till the end of the paper, we take a fixed value of 
$d \in \Z_{\ge 2}$ and 
we suppose that $(\R, \le)$ is an
ordered field satisfying $\IVT_{d^2}$ but not necessarily $\IVT$.

\subsection{$(\sigma,\tau)$-chains and Cauchy index}

A Sturm chain with respect to $I$ 
is a finite sequence of univariate polynomials
$(P_0, \dots, P_n) \in \R[X]$
such that for every $x\in I$ and $0 < i <n$, if  
 $P_i(x)=0$ then $P_{i-1}(x)P_{i+1}(x)<0$ (see \cite[Definition 3.10]{Eis}).

An important property of Sturm chains is its connection to Cauchy indices, given by Proposition \ref{prop:Cauchy_Ind_eval0} (\cite[Theorem 3.11]{Eis})

\begin{proposition}\label{prop:Cauchy_Ind_eval0}
Let 
$a, b \in \R$ with $a < b$, $I := [a, b]$,
$n \in \Z_{\ge 1}$.
If $(P_0, \dots, P_n)$ is a 
Sturm chain with respect to $I$, 
then
$$
\Ind_a^b(P_1,P_0) + 
\Ind_a^b(P_{n-1},P_n)
= 
\sum_{1 \le i \le n} 
{\rm Var}_a^b(P_{i-1},P_i)
$$
\end{proposition}

\begin{example}\label{exm:crucial0}
Here are important examples illustrating the definition of Sturm chains. 
Item b) plays a key role in the proof of the  Main Lemma (see \cite[Lemma 5.3]{Eis}) 
stating that  if a polynomial does not vanish in a rectangle, 
then the associated  winding number is zero, 
which is a key step of the algebraic-geometric 
proof of the Fundamental Theorem of Algebra in \cite{Eis}.
\begin{itemize}
\item[a)] Let $P_0,P_1 \in \R[X] \setminus \{0\}$ with $d_0:= \deg P_0 \ge 1$, $d_1 := \deg P_1  < d_0$
and $P_0, P_1$ coprime.   Let $(d_0, d_1, \dots, d_s)$ 
be the sequence of
degrees of the remainder sequence of $P_0,P_1$ in decreasing order.  We consider the classical Sturm sequence of $P_0,P_1$ : for $2 \le i \le s$, we define
$$P_{i}=-\Rem(P_{i-2},P_{i-1}).$$ 
 It is clear that the Sturm sequence
$(P_0,P_1,\ldots,P_s)$ is a Sturm chain.
\item[b)]
Given $P,Q \in \R[Y][X] \setminus \{0\}$, $p := \deg_X P \ge 1$, 
$q := \deg_X P < p$, 
$e$ the smallest even natural number greater than or 
equal to $p-q$ and $C \in \R[Y]$ the leading coefficient of $Q$.
We define  
$$\Prem (P, Q)=\Rem(C^e P, Q)\in \R[Y][X].$$
 Let $P_0,P_1 \in \R[Y][X] \setminus \{0\}$
with $d_0 := \deg_X P_0 \ge 1$, $d_1 := \deg_X P_1 < d_0$ and $P_0, P_1$ coprime  
in the unique factorization domain $\R[X, Y]$. 
Let $(d_0, d_1, \dots, d_s)$ be the sequence of
degrees in $X$ of the pseudo-remainders in decreasing order.
For $2 \le i \le s$, we define 
$$
P_i := -\Prem (P_{i-2}, P_{i-1}) \in \R[Y][X]
$$ 
and  for $1 \le i \le s$, $C_i\in \R[Y]$ the leading coefficient of $P_i$.
Since $P_0$ and $P_1$ are coprime in $\R[X,Y]$, we have that
$P_s=C_s \in \R[Y]$. 
Take an interval $[b,b']$ 
such that 
$C_1,\ldots, C_s$ have no zero on $[b, b']$. 
 We have that 
\begin{itemize}
\item for any $y \in [b, b']$, 
$(P_0(X,y), \dots, P_s(X,y)) \in \R[X]$ is a Sturm chain
with respect to $\R$,
\item for any $x \in 
\R$, 
$(P_0(x,Y), \dots, P_s(x,Y)) \in \R[Y]$ is Sturm chain
with respect to $[b, b']$.
\end{itemize}
\end{itemize}
\end{example}

We wish to use subresultants  rather than (pseudo)-remainder sequences to prove the Quantitative Main Lemma (Lemma \ref{prop:non_van_zero}), taking advantage of good degree bounds for subresultants (Proposition \ref{prop:degree_bound_res}). Unfortunately subresultants are not necessarily Sturm chains.

This is our motivation to introduce now the notion of $(\sigma, \tau)$-chain, which is a generalization 
of the 
notion
 of Sturm chain. Then, in Proposition \ref{prop:Cauchy_Ind_eval} and 
Corollary \ref{cor:good_Sturm_counting}, we develop a modified
sign changing counting rule, so that we can still use $(\sigma, \tau)$-chains
to compute Cauchy indices.

The benefit of this generalization is that
the subresultant polynomial sequence will fit in this definition for some pair 
$(\sigma, \tau)$, 
which is an essential ingredient for the proof of the Quantitative Main Lemma (Lemma \ref{prop:non_van_zero})
where we use the good degree bounds on subresultants
obtained in Proposition \ref{prop:degree_bound_res}.

\begin{defn}\label{defn:chain} Let 
$I$ be an interval of $\R$,
$n\in\Z_{\ge 1}$ and
$\sigma, \tau \in \{-1, 1\}^{n-1}$
with $\sigma = (\sigma_1, \dots, \sigma_{n-1})$ and 
$\tau = (\tau_1, \dots, \tau_{n-1})$.

A sequence of polynomials $(S_0, \dots, S_n) \in \R[X]$
is a $(\sigma, \tau)$-\emph{chain} with respect to $I$
if for $1 \le i \le n-1$ there exists polynomials $A_i, B_i, C_i \in \R[X]$
such that
\begin{enumerate}
\item \label{it:1}
$
A_iS_{i+1} + B_iS_i + C_iS_{i-1} = 0,
$
\item \label{it:2} for every $x \in I$, $\sign(A_i(x)) = \sigma_i$,
\item for every $x \in I$, $\sign(C_i(x)) = \tau_i$. 
\end{enumerate}
 
A sequence of polynomials $(S_0, \dots, S_n) \in \R[X]$
is a Sturm $(\sigma, \tau)$-chain with respect to $I$
if it is a $(\sigma, \tau)$-chain with respect to $I$
and $S_{n-1}$ and $S_n$ have no common root on $I$. 

A sequence of polynomials $(S_0, \dots, S_n) \in \R[X]$
is a \emph{good} Sturm $(\sigma, \tau)$-chain with respect to $I$
if it is a $(\sigma, \tau)$-chain with respect to $I$
and $S_n$ has no root on $I$. 

\end{defn}

Note that for $n = 1$, taking $\{-1, 1\}^0 = \{ \bullet \}$, 
any sequence $(S_0, S_1)$ in $\R[X]$ is a $(\bullet, \bullet)$-chain with 
respect to $I$.

Note also that if a sequence of polynomials $(S_0, \dots, S_n)$ in $\R[X]$
is a $(\sigma, \tau)$-chain with respect to $I$, then for every
$0 \le m \le n-1$,  $(S_m, \dots, S_n)$ is a 
$(\sigma', \tau')$-chain with respect to $I$, 
with $\sigma' := (\sigma_{m+1}, \dots, \sigma_{n-1})$ and
$\tau' := (\tau_{m+1}, \dots, \tau_{n-1})$.
The analogous statements hold also for Sturm $(\sigma, \tau)$-chains and
good Sturm $(\sigma, \tau)$-chains. 

\begin{example}\label{exm:crucial}
Here are important examples illustrating the definition of $(\sigma, \tau)$-chains. 
Item b) plays a key role in the proof of our Quantitative Main (Lemma \ref{prop:non_van_zero})  which is a crucial step in the proof of our Quantitative Fundamental Theorem of Algebra (Theorem \ref{th:main}).
\begin{itemize}
\item[a)] Let $S_0,S_1 \in \R[X] \setminus \{0\}$ with $d_0 := \deg S_0 \ge 1$, $d_1 := \deg S_1  < d_0$
and $S_0, S_1$ coprime.  Let $(d_0, d_1, \dots, d_s)$ 
be the sequence of
degrees of the non-defective subresultant polynomials of $S_0,S_1$ in decreasing order
(note that $d_0 = p$, $d_1 = q$), 
and $d_{-1} := p+1$. Finally, for $2 \le i \le s$, we define 
$$
S_i := \sResP_{d_{i-1}-1}(S_0, S_1) \in \R[X]
$$
(note that the above identity also holds for $i = 0, 1$). 
Since $S_0$ and $S_1$ are coprime in $\R[X]$, by the Structure Theorem of 
Subresultants (Theorem \ref{thm:structure_thm_subresultants}) we have that
$S_s \in \R$ and $d_s=0$. Also, defining
for $1 \le i \le s-1$
$$
\begin{array}{rcl}
A_i & := &
T_{d_{i-2}-1}(S_0, S_1) \cdot 
{\sRes}_{d_{i-1}}(S_0, S_1) \in \R \setminus \{0\}
\\[2mm]
B_i & := & -{\rm Quot}
\left(
T_{d_{i-1}-1}(S_0, S_1) \cdot 
{\sRes}_{d_{i}}(S_0, S_1)
\cdot S_{i-1}
 , 
S_i 
\right) 
\in \R[X], \\[2mm]
C_i & := & T_{d_{i-1}-1}(S_0, S_1) \cdot 
{\sRes}_{d_{i}}(S_0, S_1) \in \R \setminus \{0\}, 
\end{array}
$$
we have 
$$
A_iS_{i+1} + B_iS_i + C_iS_{i-1} = 0.
$$
We define $\sigma_i := \sign(A_i)$
and $\tau_i := \sign(C_i)$ for $1 \le i \le s-1$ and we have that 
$(S_0,S_1,\ldots,S_s)$ is a good Sturm $(\sigma,\tau)$-chain.

\item[b)] Let $S_0,S_1 \in \R[X, Y] \setminus \{0\}$
with $p := \deg_X S_0 \ge 1$, $q := \deg_X S_1 < p$ and $S_0, S_1$ coprime  
in the unique factorization domain $\R[X, Y]$. 
Let $(d_0, d_1, \dots, d_s)$ be the sequence of
degrees of the non-defective subresultant polynomials in decreasing order
(note that $d_0 = p$ and $d_1 = q$), 
and $d_{-1} := p+1$; where all the subresultants are 
defined considering $X$ as the main variable (this is to say, $\D = \R[Y]$ in 
Definition \ref{def:subres}).
Finally, for $2 \le i \le s$, we define 
$$
S_i := \sResP_{d_{i-1}-1}(S_0, S_1) \in \R[Y][X]
$$
(note that the above identity also holds for $i = 0, 1$). 
Since $S_0$ and $S_1$ are coprime in $\R[X,Y]$, by the Structure Theorem of 
Subresultants (Theorem \ref{thm:structure_thm_subresultants}) we have that
$S_s \in \R[Y]$. Also,
defining 
for $1 \le i \le s-1$
$$
\begin{array}{rcl}
A_i & := &
T_{d_{i-2}-1}(S_0, S_1) \cdot 
{\sRes}_{d_{i-1}}(S_0, S_1) \in \R[Y] \setminus\{0\}
\\[2mm] 
B_i & := & -{\rm Quot}
\left(
T_{d_{i-1}-1}(S_0, S_1) \cdot 
{\sRes}_{d_{i}}(S_0, S_1)
\cdot S_{i-1}
 , 
S_i 
\right) 
\in \R[Y][X], \\[2mm]
C_i & := & T_{d_{i-1}-1}(S_0, S_1) \cdot 
{\sRes}_{d_{i}}(S_0, S_1) \in \R[Y] \setminus \{0\},
\end{array}
$$
we have
$$
A_iS_{i+1} + B_iS_i + C_iS_{i-1} = 0.
$$
Take an interval $[b,b']$ 
such that 
$S_s$, $A_i$ and $C_i$, for $i=1,\ldots,s-1$, have constant sign different from $0$ on $[b, b']$. 
We  define then $\sigma = (\sigma_1, \dots, \sigma_{s-1})$, 
$\tau = (\tau_1, \dots, \tau_{s-1}) \in \{-1, 1\}^{s-1}$
by choosing any $c \in [b, b']$ and taking $\sigma_i := \sign(A_i(c))$
and $\tau_i := \sign(C_i(c))$ for $1 \le i \le s-1$.
In this way, since also $S_s \in \R[Y]$ does not vanish on $[b, b']$, we have that 
\begin{itemize}
\item for any $y \in [b, b']$, 
$(S_0(X,y), \dots, S_s(X,y)) \in \R[X]$ is a good Sturm $(\sigma, \tau)$-chain
with respect to $\R$,
\item for any $x \in 
\R$, 
$(S_0(x,Y), \dots, S_s(x,Y)) \in \R[Y]$ is a good Sturm $(\sigma, \tau)$-chain
with respect to $[b, b']$.
\end{itemize}
\end{itemize}
\end{example}

\begin{lemma}\label{lem:no_common_roots}
Let 
$I$ be an interval of $\R$,
$n\in\Z_{\ge 1}$ and $\sigma, \tau \in \{-1, 1\}^{n-1}$.
If a sequence of polynomials $(S_0, \dots, S_n)$ in $\R[X]$
is a Sturm $(\sigma, \tau)$-chain with respect to $I$, then for every
$1 \le m \le n$, $S_{m-1}$ and $S_m$ have no common root on $I$.
\end{lemma}

\begin{proof}{Proof:}
The proof can be easily done by reverse induction on $m = n, \dots, 1$, 
taking into account that conditions \ref{it:1} and \ref{it:2} from 
Definition \ref{defn:chain} imply that for $m < n$, 
any common root of $S_{m-1}$ and $S_m$ 
would also be a root of $S_{m+1}$.
\end{proof}

We introduce some more useful notation.

\begin{notn}\label{not:V}
Let $a, b \in \R$, 
$n \in \Z_{\ge 1}$, 
$(S_0, \dots, S_n)$ in $\R[X]$
and $\sigma, \tau \in \{-1, 1\}^{n-1}$. We define
$$
\epsilon(\sigma,\tau)_i := \prod_{1 \le j \le i-1} \sigma_j \tau_j
$$
for $1 \le i \le n$ and
$$
{\rm Var}(
\sigma,\tau)_a^b(S_0, \dots, S_n) := \sum_{1 \le i \le n} 
\epsilon(\sigma,\tau)_i{\rm Var}_a^b(S_{i-1},S_i).
$$
\end{notn}

Note that it is always the case that $\epsilon(\sigma,\tau)_1 = 1$. 
.

\begin{proposition}\label{prop:Cauchy_Ind_eval}
Let 
$a, b \in \R$ with $a < b$, $I := [a, b]$,
$n \in \Z_{\ge 1}$ and  $\sigma, \tau \in \{-1, 1\}^{n-1}$.
If $(S_0, \dots, S_n)$ is a 
Sturm $(\sigma, \tau)$-chain with respect to $I$ and $\deg S_0, \dots, \deg S_n \le
d^2$, 
then
$$
\Ind_a^b(S_1,S_0) + 
\epsilon(\sigma,\tau)_n\Ind_a^b(S_{n-1},S_n)
= 
{\rm Var}(\sigma,\tau)_a^b(S_0, \dots, S_n).
$$
\end{proposition}

Note that the identity in Proposition \ref{prop:Cauchy_Ind_eval0} (\cite[Theorem 3.11]{Eis}) is exactly the identity
in 
Proposition \ref{prop:Cauchy_Ind_eval} in the particular case $\sigma = \tau = (1, 1, \dots, 1)$

\begin{proof}{Proof of Proposition \ref{prop:Cauchy_Ind_eval}:} 
By Lemma \ref{lem:no_common_roots}, 
we know that for every
$1 \le m \le n$, $S_{m-1}$ and $S_m$ have no common root on $I$.
We proceed then by induction on $n$. For $n = 1$, the result holds by Proposition 
\ref{prop:inv_formula}.

Now we take $n \ge 2$. 
Let $x$ be a root of $S_1$ on $I$ (and therefore $x$ is not a root neither of $S_0$ nor of $S_2$). The identity 
$$
A_1S_2 + B_1S_1 + C_1S_0 = 0
$$
implies that 
$ C_1(x)S_0(x) = -A_1(x)S_2(x) \ne 0$ and then 
$  \sign(S_0(x)) = -\sigma_1 \tau_1\sign(S_2(x))$. 
From this we deduce
$$
\Ind_a^b(S_0,S_1) = -\sigma_1 \tau_1\Ind_a^b(S_2,S_1).
$$

We consider
$\sigma' := (\sigma_2, \dots, \sigma_{n-1})$, 
$\tau' := (\tau_2, \dots, \tau_{n-1})$
and we apply the inductive hypothesis to the Sturm $(\sigma', \tau')$-chain 
$(S_1, \dots, S_n)$. 
For $2 \le i \le n$ we have that 
$\epsilon(\sigma,\tau)_i = \sigma_1 \tau_1 \epsilon(\sigma',\tau')_{i-1}$. 

Finally, using Proposition \ref{prop:inv_formula},
$$
\begin{array}{cl}
& \Ind_a^b(S_1,S_0) + 
\epsilon(\sigma,\tau)_n\Ind_a^b(S_{n-1},S_{n}) \\[2mm]
=  &  \Ind_a^b(S_1,S_0) + 
\Ind_a^b(S_0,S_1) 
+ \sigma_1\tau_1 \Ind_a^b(S_2,S_1) +
\sigma_1\tau_1 \epsilon(\sigma',\tau')_{n-1}\Ind_a^b(S_{n-1},S_{n}) \\[2mm]
 = &{\rm Var}_a^b(S_0,S_1)
+ \sigma_1\tau_1 {\rm Var}(\sigma',\tau')_a^b(S_1, \dots, S_n)\\[2mm]
 = & {\rm Var}(\sigma,\tau)_a^b(S_0, \dots, S_n) 
\end{array}
$$
as we wanted to prove. 
\end{proof}

\begin{corollary}\label{cor:good_Sturm_counting}
Let 
$a, b \in \R$ with $a < b$, $I := [a, b]$,
$n \in \Z_{\ge 1}$ and  $\sigma, \tau \in \{-1, 1\}^{n-1}$. 
If $(S_0, \dots, S_n)$ is a 
good Sturm $(\sigma, \tau)$-chain with respect to $I$ and $\deg S_0, \dots, \deg S_n \le
d^2$, 
then
$$
\Ind_a^b(S_1,S_0) 
= {\rm Var}(\sigma, \tau)_a^b(S_0, \dots, S_n).
$$
\end{corollary}  
\begin{proof}{Proof:}
Since  $(S_0, \dots, S_n)$ is a 
good Sturm $(\sigma, \tau)$-chain with respect to $I$, $S_n$ has no roots on $I$ and 
$$ \Ind_a^b(S_{n-1},S_{n}) = 0,$$
therefore the claim holds by Proposition \ref{prop:Cauchy_Ind_eval}, 
\end{proof}

\subsection{Quantitative Main Lemma}
Our next goal is to 
prove a 
quantitative adaptation 
 of the Main Lemma (see \cite[Lemma 5.3]{Eis}): using
 $\IVT_{d^2}$ and subresultants we want to prove that
if 
 $F \in \C[X, Y]$ with $\deg F\le d$ does not vanish on $\Gamma$,
then $w(F \, | \, \partial \Gamma) = 0$. 

In order to be able to work with subresultants, we need to consider
separately for $F \in \C[X, Y]$
the degrees with respect to $X$ and $Y$ of 
$F_{\rm re}$
 and 
 $F_{\rm im},$ 
and we wish each of these two degrees of
$F_{\rm re}$
to drop with respect
to the respective  degree of
$F_{\rm im}$.
Since
$$
(iF)_{\rm re} = - F_{\rm im} \ \ \hbox{ and } \ \ (iF)_{\rm im} = F_{\rm re}, 
$$
up to multiplication by $i$, 
it will be enough for our purposes if these degrees are different. 
We will also need some degree control on some auxiliary subresultant polynomials
which will 
play a key role in our proof.
For these reasons, we introduce the 
following definition.

\begin{defn}\label{defn:well_controlled}
Let $F \in \C[X, Y]$.
We say that $F$ is \emph{well-controlled} if
the following conditions are satisfied:
\begin{enumerate}
\item 
$F_{\rm re}, F_{\rm im} \ne 0$,
\item $\deg_X F_{\rm im} \ne \deg_X F_{\rm re}$ and $\deg_Y F_{\rm im} \ne \deg_Y F_{\rm re}$. 
\end{enumerate}
For a well-controlled $F$, we denote by $F^X$ the unique polynomial in $\{F, iF\} \subset \C[X, Y]$ 
such that $\deg_X  F^X_{\rm im} > \deg_X F^X_{\rm re}$.
Similarly, we denote by $F^Y$ the unique polynomial in $\{F, iF\} \subset \C[X, Y]$ 
such that 
$\deg_Y  F^Y_{\rm im} > \deg_Y F^Y_{\rm re}$.
\end{defn}

\begin{example} \label{exmple:monic_well_controlled}
Let $F \in \C[Z] \setminus \C$ with $F$ monic.
Then $F$ is well-controlled. 
\end{example}

Now we are ready to prove our 
quantitative version of the Main Lemma \cite[Lemma 5.3]{Eis}. 
We stress the fact that it is
in the proof of the Quantitative Main Lemma (Lemma \ref{prop:non_van_zero}) that
subresultant polynomials play a key role to keep control of the 
degree of intermediate auxiliary polynomials.
We need to use the hypothesis $\IVT_{d^2}$ for proving the result for a polynomial of degree $\le d$. 

\begin{lemma}(Quantitative Main Lemma)\label{prop:non_van_zero}
Suppose that $(\R, \le)$ is an
ordered field satisfying $\IVT_{d^2}$.
Let $\Gamma := [x_0, x_1] \times [y_0, y_1] \subset \R^2$ and $F\in \C[X,Y]$
with $\deg F\le d$ and
such that $F$ does not vanish on $\Gamma$.
If $F$ is well-controlled, then $w(F \, | \, \partial \Gamma) = 0$. 
\end{lemma}

\begin{proof}{Proof:} 
We will produce in several steps
a suitable grid partition of $\Gamma$ into a finite number of rectangles 
$\Gamma_1, \dots, \Gamma_s$ and we will prove that $w(F \, | \, \partial \Gamma_i) = 0$
for $1 \le i \le s$. Then the result  will follow from Lemma \ref{lem:grid_partition}.
Let  
$G$ be
a greatest common divisor 
of 
$F_{\rm re}$ and $F_{\rm im}$
 in the unique factorization domain $\R[X, Y]$
and let $F^X, F^Y \in \C[X, Y]$ be as in Definition \ref{defn:well_controlled}.

\textit{First step.} 
We decompose
$
F^X_{\rm im} = G S_0$ and $F^X_{\rm re} = G S_1$ and 
note that 
we have  $\deg S_0, \deg S_1 \le d$ and 
$$d\ge p := \deg_X S_0 > q := \deg_X S_1.$$

We consider $S_0$ and $S_1$ as elements of $\R[Y][X]$ and
we take the subresultant polynomial sequence
with respect to 
the variable $X$ (as in Section \ref{sec:subres}, taking $\D = \R[Y]$)
$$
\sResP_{p}(S_0, S_1) = S_0, \sResP_{p-1}(S_0, S_1) = S_1, 
\dots, \sResP_{0}(S_0, S_1) \in \R[Y][X].
$$
For $0 \le j \le p$ and 
$0 \le i \le j$, the degree in $Y$ of the coefficient of $X^i$ 
in ${\sResP}_j(S_0, S_1) \in \R[Y][X]$ is bounded by $d^2$.
This is so by Proposition \ref{prop:degree_bound_res} for $
0 \le j \le q$ and by definition of subresultant polynomials 
for $q < j \le p$.

We take $(d_0, d_1, \dots, d_s)$ as the sequence of
degrees of the non-defective subresultant polynomials in decreasing order
(note that $d_0 = p$ and $d_1 = q$), 
and $d_{-1} := p+1$, and we define  $\cY_1$ as the  union of the sets of roots in $[y_0, y_1]$ of the polynomials 
$S_s, T_{d_{-1}-1}(S_0, S_1), \dots,$ $T_{d_{s-2}-1}(S_0, S_1), 
{\sRes}_{d_{0}}(S_0, S_1), \dots,$ ${\sRes}_{d_{s-1}}(S_0, S_1)
\in \R[Y] \setminus\{0\}$
(note that $T_{d_{-1}-1}(S_0, S_1)$ $=$ ${\sRes}_{d_{0}}(S_0, S_1)$ $=$ $1$, 
therefore these two polynomials actually add no roots to the set $\cY_1$; note also 
that in the particular case $s = 1$, $\cY_1$ is just the set of roots in 
$[y_0, y_1]$ of the polynomial $S_1$).

For uniformity reasons in exposition we define $\cY_2 := \{y_0, y_1\}$
and we also define 
$$
\cY_3 := \{y \in [y_0, y_1] \ | \ F^X_{\rm re}(X,y) = 0 \in \R[X] 
\hbox{ or } F^X_{\rm im}(X,y) = 0 \in \R[X] \}.
$$
Finally, we define
$$
\cY := \cY_1 \cup \cY_2 \cup \cY_3.
$$
We think of $\cY$ as the set of the $Y$-coordinates of bad behaving
points in $\Gamma$. 
Suppose $\cY = \{b_1, \dots, b_{\ell}\}$ with $y_0 = b_1 < \dots < b_{\ell} = y_1$.

\textit{Second step.} We proceed as in the first step, but replacing polynomial 
$F^X$ by $F^Y$ and the roles of variables $X$ and $Y$, to produce a
set $\cX \subset [x_0, x_1]$, which we think of as the set of the 
the $X$-coordinates of bad behaving
points in $\Gamma$. 
Suppose $\cX = \{a_1, \dots, a_{k}\}$ with $x_0 = a_1 < \dots < a_{k} = x_1$.

\textit{Third step.}
We take $\cZ := \cX \times \cY \subset \Gamma$. 
For each $z =(a,b) \in \cZ$, since $F(a,b) \ne 0$, 
by Proposition \ref{prop:non_van_zero_local}
there exist $\delta_z > 0$ such that the winding number of $F$ vanishes on 
any rectangle contained in $[a - \delta_z, a + \delta_z] \times  
[b - \delta_z, b + \delta_z]$. So we take $\delta > 0$, with $\delta \le \delta_z$
for every $z \in \cZ$ and such that
$$
x_0 = a_1 < a_1 + \delta < a_2 - \delta < a_2 + \delta < a_3-\delta < \dots 
< a_{k-1} + \delta < a_k - \delta < a_k = x_1
$$
and
$$
y_0 = b_1 < b_1 + \delta < b_2 - \delta < b_2 + \delta < b_3-\delta < \dots 
< b_{\ell-1} + \delta < b_{\ell} - \delta < b_{\ell} = y_1.
$$
We divide intervals $[x_0, x_1]$ and $[y_0, y_1]$ using 
all these numbers above, and finally we use 
these divisions of these intervals to obtain a grid partition of $\Gamma
= [x_0, x_1] \times [y_0, y_1]$.

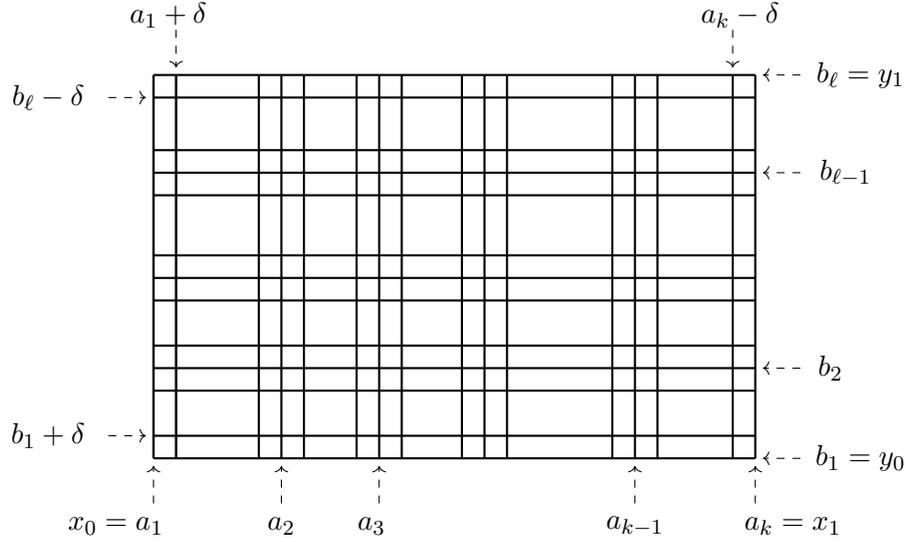
\begin{center}
\begin{tikzpicture}
      \draw[line width=0.9pt,-] (-4,-2.3) -- (-4,2.8);
      \draw[line width=0.9pt,-] (-3.7,-2.3) -- (-3.7,2.8);
      
      \draw[line width=0.8pt,-] (-2.6,-2.3) -- (-2.6,2.8);
      \draw[line width=0.8pt,-] (-2.3,-2.3) -- (-2.3,2.8);
      \draw[line width=0.8pt,-] (-2.0,-2.3) -- (-2.0,2.8);
      
      \draw[line width=0.8pt,-] (-1.3,-2.3) -- (-1.3,2.8);
      \draw[line width=0.8pt,-] (-1.0,-2.3) -- (-1.0,2.8);
      \draw[line width=0.8pt,-] (-0.7,-2.3) -- (-0.7,2.8);
      
      \draw[line width=0.8pt,-] (0.1,-2.3) -- (0.1,2.8);
      \draw[line width=0.8pt,-] (0.4,-2.3) -- (0.4,2.8);
      \draw[line width=0.8pt,-] (0.7,-2.3) -- (0.7,2.8);
      
      \draw[line width=0.8pt,-] (2.1,-2.3) -- (2.1,2.8);
      \draw[line width=0.8pt,-] (2.4,-2.3) -- (2.4,2.8);
      \draw[line width=0.8pt,-] (2.7,-2.3) -- (2.7,2.8);
            
      \draw[line width=0.8pt,-] (3.7,-2.3) -- (3.7,2.8);
      \draw[line width=0.8pt,-] (4,-2.3) -- (4,2.8);

     \node at (-4.5,-3.2) {$x_0 = a_1$};
     \draw[->, dashed] (-4, -2.9) -- (-4, -2.4);

     \node at (-3.8,3.6) {$a_1 + \delta $};
     \draw[->, dashed] (-3.7, 3.4) -- (-3.7, 2.9);

     \node at (-2.3,-3.2) {$a_2$};
     \draw[->, dashed] (-2.3, -2.9) -- (-2.3, -2.4);

     \node at (-1.1,-3.2) {$a_3$};
     \draw[->, dashed] (-1.0, -2.9) -- (-1.0, -2.4);
     
     \node at (2.4,-3.2) {$a_{k-1}$};
     \draw[->, dashed] (2.4, -2.9) -- (2.4, -2.4);

     \node at (3.8,3.6) {$a_k - \delta $};
     \draw[->, dashed] (3.7, 3.4) -- (3.7, 2.9);

     \node at (4.5,-3.2) {$a_k = x_1$};
     \draw[->, dashed] (4, -2.9) -- (4, -2.4);

      \draw[line width=0.8pt,-] (-4,-2.3) -- (4,-2.3);
      \draw[line width=0.8pt,-] (-4,-2) -- (4,-2);
     
      \draw[line width=0.8pt,-] (-4, -1.4) -- (4,-1.4);
      \draw[line width=0.8pt,-] (-4, -1.1) -- (4,-1.1);
      \draw[line width=0.8pt,-] (-4, -0.8) -- (4,-0.8);
      
      \draw[line width=0.8pt,-] (-4, -0.2) -- (4,-0.2);
      \draw[line width=0.8pt,-] (-4, 0.1) -- (4,0.1);
      \draw[line width=0.8pt,-] (-4, 0.4) -- (4,0.4);
      
      \draw[line width=0.8pt,-] (-4, 1.2) -- (4,1.2);
      \draw[line width=0.8pt,-] (-4, 1.5) -- (4,1.5);
      \draw[line width=0.8pt,-] (-4, 1.8) -- (4,1.8);
      
      \draw[line width=0.8pt,-] (-4, 2.5) -- (4,2.5);
      \draw[line width=0.8pt,-] (-4, 2.8) -- (4,2.8);
     
     \node at (5.4,-2.3) {$b_1 = y_0$};
     \draw[->, dashed] (4.6, -2.3) -- (4.1, -2.3);
     
     \node at (-5.4,-2) {$b_1 + \delta$};
     \draw[->, dashed] (-4.6, -2) -- (-4.1, -2);

     \node at (5.0,-1.1) {$b_2$};
     \draw[->, dashed] (4.6, -1.1) -- (4.1, -1.1);
     
     \node at (5.2,1.5) {$b_{\ell - 1}$};
     \draw[->, dashed] (4.6, 1.5) -- (4.1, 1.5);
     
     \node at (-5.4,2.5) {$b_\ell - \delta $};
     \draw[->, dashed] (-4.6, 2.5) -- (-4.1, 2.5);

     \node at (5.4,2.8) {$b_\ell = y_1$};
     \draw[->, dashed] (4.6, 2.8) -- (4.1, 2.8);

\end{tikzpicture}
\end{center}

Now that the grid partition is defined, we have to prove that the winding number
of $F$ vanishes on each rectangle in the grid. 
Take $\Gamma' = 
[a,a']\times [b,b']\subset \Gamma$ as
one of this rectangles. Then either 
there is a single point of $\cZ$ in $\Gamma'$ 
or there is no point of $\cZ$ in 
$\Gamma'$. In the first case, $w(F \, | \, \partial\Gamma') = 0$ by Proposition \ref{prop:non_van_zero_local}
and the choice of $\delta$. In the second case, then either 
$[a, a'] \cap \cX = \emptyset$ or $[b, b'] \cap \cY = \emptyset$.

Let us suppose first that $[b, b'] \cap \cY = \emptyset$ holds
and prove that $w(F^X  \, | \,  \partial\Gamma') = 0$; then
$w(F  \, | \,  \partial\Gamma') = 0$ as well
either because $F^X = F$ or because $F^X = iF$ and by Proposition \ref{prop:mult_wind}.

Since $F$ is well-controlled 
of degree at most $d$
and the polynomials $T_{d_{-1}-1}(S_0, S_1), \dots,$ $T_{d_{s-2}-1}(S_0, S_1)$, 
${\sRes}_{d_{0}}(S_0, S_1), \dots,$ ${\sRes}_{d_{s-1}}(S_0, S_1)$
$\in \R[Y]$ are coefficients of subresultant polynomials of $S_0$ and $S_1$
with respect to variable $X$, their degree in $Y$
is bounded
by $d^2$. 
Since none of these polynomials vanishes on $[b, b']$, 
using the notation and results from Example \ref{exm:crucial} b), 
for $1 \le i \le s-1$ we have that 
$A_i$ and $C_i$ have constant sign different from $0$ on $[b, b']$
and
\begin{itemize}
\item for any $y \in [b, b']$, 
$(S_0(X,y), \dots, S_s(X,y)) \in \R[X]$ is a good Sturm $(\sigma, \tau)$-chain
with respect to $[a, a']$ with all its elements with degree bounded by $d \le d^2$,
\item for any $x \in [a, a']$, 
$(S_0(x,Y), \dots, S_s(x,Y)) \in \R[Y]$ is a good Sturm $(\sigma, \tau)$-chain
with respect to $[b, b']$ with all its elements with degree bounded by $d^2$.
\end{itemize}
Taking into account that 
$$
0 \ne F^X_{\rm re}(X, b),F^X_{\rm im}(X, b),
F^X_{\rm re}(X, b'),F^X_{\rm im}(X, b')   \in \R[X],$$
$$
0 \ne F^X_{\rm re}(a, Y),F^X_{\rm im}(a, Y),
F^X_{\rm re}(a', Y),F^X_{\rm im}(a', Y)  \in \R[Y],$$
we conclude that 
$$\begin{array}{rcl}
2w(F^X \, | \, \partial \Gamma')
&= &
\Ind_{a}^{a'}(F^X_{\rm re}(X, b),F^X_{\rm im}(X,b))
+
\Ind_{b}^{b'}
(F^X_{\rm re}(a',Y),F^X_{\rm im}(a',Y))
\\[2mm]
&&
+
\Ind_{a'}^{a}(F^X_{\rm re}(X, b'),F^X_{\rm im}(X,b'))
+
\Ind_{b'}^{b}
(F^X_{\rm re}(a,Y),F^X_{\rm im}(a,Y))\\[2mm]
&
=
&
\Ind_{a}^{a'}
(S_1(X,b),S_0(X, b))
+
\Ind_{b}^{b'}
(S_1(a',Y),S_0(a',Y))
\\[2mm]
&&
+
\Ind_{a'}^{a}
(S_1(X,b'),S_0(X, b'))
+
\Ind_{b'}^{b} 
(S_1(a,Y),S_0(a,Y))
\\[2mm]
&
=& 
{\rm Var}(\sigma,\tau)_{a}^{a'}(S_0(X, b), \dots, S_s(X, b))
+
{\rm Var}(\sigma,\tau)_{b}^{b'}(S_0(a', Y), \dots, S_s(a', Y))
\\[2mm]
&&
+
{\rm Var}(\sigma,\tau)_{a'}^{a}(S_0(X, b'), \dots, S_s(X, b'))
+
{\rm Var}(\sigma,\tau)_{b'}^{b}(S_0(a, Y), \dots, S_s(a, Y))\\[2mm]
&=& 0
\end{array}
$$
using Corollary \ref{cor:good_Sturm_counting}. 

In case that $[a, a'] \cap \cX = \emptyset$ holds, we proceed in a similar
way
exchanging the roles of $X$ and $Y$,
to prove that 
$w(F^Y \, | \, \partial \Gamma')
= 0$, and then we have that 
$w(F \, | \, \partial \Gamma')
= 0$ again either because $F^Y = F$ or because $F^Y = iF$ and by Proposition \ref{prop:mult_wind}.
\end{proof}

\subsection{The winding number counts the complex roots}

From Example \ref{exm:wind_linear}, Proposition \ref{prop:mult_wind} and the 
Quantitative
Main Lemma \ref{prop:non_van_zero} 
we deduce the following result.

\begin{theorem}\label{thm:wind_counts}
Let $\Gamma \subset \R^2$ be a rectangle and $F\in \C[Z]$ with
$\deg F \le d$ and such that $F$ does not vanish in $\partial \Gamma$. Then 
$w(F \, | \, \partial \Gamma)$ counts the number of zeros of $F$ in the 
interior of $\Gamma$ 
with multiplicity. 
\end{theorem}

\begin{proof}{Proof:} Factorize $F = a \cdot (Z - z_1) \cdot \dots (Z-z_r) 
\cdot \tilde F$ with $a \in \C$, $z_1, \dots, z_r \in \Gamma \setminus \partial \Gamma$
and monic $\tilde F \in \C[Z]$ with no roots in $\Gamma$. 
If $\tilde F = 1$ the result follows from Example \ref{exm:wind_linear} and
Proposition \ref{prop:mult_wind}. Otherwise, since
$\tilde F$ is well-controlled (Example \ref{exmple:monic_well_controlled}) 
the result follows from Example \ref{exm:wind_linear} and
Proposition \ref{prop:mult_wind} and  the Main Lemma \ref{prop:non_van_zero}.
\end{proof}

\subsection{Quantitative Homotopy}

The last ingredient for the proof of Theorem \ref{th:main} is 
a
quantitative homotopy tool
similar to
\cite[Theorem 5.4 and Corollary 5.5, Proposition 5.8 and
Theorem 5.9]{Eis}. Since we need to deal with well-controlled 
polynomials, we have to divide the homotopy in two steps, one for the real 
part and one for the imaginary part.

\begin{theorem}\label{thm:almost_main}
Let $F\in \C[Z]$, with $F \ne 0$ and $\deg F = e \le d$. 
There exists $r \in \R$, $r > 0$ 
such that if $m \ge r$ and $\Gamma := [-m, m] \times [-m, m]$, 
then $w(F \, | \, \partial \Gamma) = e$. 
\end{theorem}

\begin{proof}{Proof:} 
If $e = 0$ there is nothing to prove, so we suppose $e > 0$.
By Proposition \ref{prop:mult_wind}, we can also suppose that $F$
is monic. Let
$$
F = Z^e + \sum_{j = 0}^{e-1} (a_j+ ib_j)Z^j
$$
with $a_j, b_j \in \R$ for $0 \le j \le e-1$ and take $G := F - Z^e \in \C[Z]$
collecting all the terms of degree less than $e$ in $F$. 

We take the auxiliary polynomial 
$$
K := (X + iY)^e + G_{\rm re} \in \C[X, Y].
$$
Note that in general $K$ does not come from a polynomial in $\C[Z]$
by means of the substitution $Z = X+iY$.
The idea of the proof is to obtain $r \in \R$, $r >0$
such that if $m \ge r$ and $\Gamma = [-m, m] \times [-m, m]$, 
then $e = w(Z^e \, | \, \partial \Gamma) = 
w(K \, | \, \partial \Gamma)
= w(F \, | \, \partial \Gamma)$ (see Example
\ref{exm:wind_monomial}).

We suppose  $G_{\rm re}, G_{\rm im} \ne 0$, and if this is not the case, the rest of the 
proof can 
be simplified. Actually, the only case where $G_{\rm re} = 0$ or $G_{\rm im} = 0$ is 
$G \in \R \cup i \R$.

We define $H_0, H_1 \in \C[X, Y, T]$ by
$$
H_0(X, Y, T) 
:= 
(1-T)(X+iY)^e + TK(X, Y) 
=
(X + iY)^e + T \, G_{\rm re}(X,Y)
$$
and
$$
H_1(X, Y, T) 
:= 
(1-T)K(X, Y) + TF(X + i Y)
=
(X + iY)^e + G_{\rm re}(X,Y) + i \,  T \, G_{\rm im}(X,Y). 
$$

Take $r := 1 + 2\max\{|a_j + ib_j| \, | \, 0 \le j \le e-1\}$. 
We proceed similarly to  \cite[Proposition 5.8]{Eis} to prove that
both $H_0,  H_1$ do not vanish on $\partial \Gamma \times [0, 1]$. 
For $k = 0, 1$, and $(x, y, t) \in \partial \Gamma \times [0, 1]$, 
we have
$$
\begin{array}{cl}
& |H_k(x, y, t) - (x+iy)^e| \\[3mm]
\le  &  2\displaystyle{\Big| \sum_{j = 0}^{e-1} (a_j+ ib_j)(x + iy)^j \Big|} \\[6mm]
\le  &  2\displaystyle{ \sum_{j = 0}^{e-1} |a_j+ ib_j| |x + iy|^j} \\[6mm]
\le  &  (r-1)\displaystyle{ \sum_{j = 0}^{e-1}  |x + iy|^j} \\[6mm]
\le  & |x+iy|^e -1. \cr
\end{array}
$$
Therefore
$$
|H_k(x, y, t)| \ge |(x+iy)^e| - |H_k(x, y, t) - (x+iy)^e| \ge |x+iy|^e - |x+iy|^e + 1 = 1. 
$$

Now enlarge $r$ if necessary so that $X \pm m$ and $Y \pm m$ are not factors 
of $F_{\rm re}(X, Y), G_{\rm re}(X, Y)$ and $G_{\rm im}(X, Y)$.
Then it can be verified that the polynomials $H_0(X, -m, T), 
H_0(X, m, T), H_1(X, -m, T),$ $H_1(X, m, T)$ $\in \C[X, T]$
and $H_0(-m, Y, T), 
H_0(m, Y, T), H_1(-m, Y, T), H_1(m, Y, T) \in \C[Y, T]$
are of degree at most $d$ 
and well-controlled. 

Finally, take $\Gamma_X := [-m, m] \times [0, 1]$, $\Gamma_Y := [-m, m] \times [0, 1]$ and $\Gamma_T := \Gamma$. 
By the Main Lemma \ref{prop:non_van_zero},
$$
w(H_0(X, -m, T) \, | \, \partial \Gamma_Y) = 
w(H_0(X, m, T) \, | \, \partial \Gamma_Y) 
=
w(H_1(X, -m, T) \, | \, \partial \Gamma_Y) =
w(H_1(X, m, T) \, | \, \partial  \Gamma_Y) =0
$$
and
$$
w(H_0(-m, Y, T) \, | \, \partial \Gamma_X) =
w(H_0(m, Y, T) \, | \, \partial \Gamma_X) =
w(H_1(-m, Y, T) \, | \, \partial \Gamma_X) =
w(H_1(m, Y, T) \, | \, \partial  \Gamma_X) = 0.
$$
Therefore, 
by Lemma \ref{lem:homot_cube} applied to $H_0$ and $H_1$ and the cube $[-m,m] 
\times [-m,m] \times [0,1] \subset \R^3$, we have
$$
e=w(Z^e \, | \, \partial \Gamma) 
= 
w(H_0(X, Y, 0) \, | \, \partial \Gamma_T) 
=
w(H_0(X, Y, 1) \, | \, \partial \Gamma_T) 
=
w(K \, | \, \partial \Gamma) 
$$
and
$$
w(K \, | \, \partial \Gamma) 
=
w(H_1(X, Y, 0) \, | \, \partial \Gamma_T) 
=
w(H_1(X, Y, 1) \, | \, \partial \Gamma_T) 
=
w(F \, | \, \partial \Gamma)
$$
as we wanted to prove. 
\end{proof}

\subsection{Proof of Theorem \ref{th:main}}
We are now ready to deduce our main result.

\begin{proof}{Proof of Theorem \ref{th:main}:}
As mentioned before, since $\FTA_1$ holds even under no assumptions on $(\R, \le)$ 
we suppose  $d \ge 2$.
Take $F \in \C[Z] \setminus \C$ with $\deg F = e \le d$. 
By Theorem \ref{thm:almost_main},
there exists $r \in \R$, $r > 0$ such that if $m \ge r$ and $\Gamma := [-m, m] 
\times [-m, m]$, 
then $w(F \, | \, \partial \Gamma) = e \ge 1$.
By Theorem \ref{thm:wind_counts},
$w(F \, | \, \partial \Gamma)$ counts the number of zeros of $F$ in the 
interior of $\Gamma$ with multiplicity. 
This implies that there exists at least one $z \in \Gamma \subset \R^2 \sim \C$
such that $F(z) = 0$. 
\end{proof}

\textbf{Acknowledgement:} We would like to express our gratitude to the anonymous referees whose suggestions
helped us to improve the readability of the paper.

\end{document}